\documentclass[12pt, reqno]{amsart}
\usepackage{mathrsfs}
\usepackage{amssymb,amsthm,amsmath}
\usepackage[numbers,sort&compress]{natbib}
\usepackage{amssymb,amsmath}
\usepackage{amsfonts}
\usepackage{mathrsfs}
\usepackage{latexsym}
\usepackage{amssymb}
\usepackage{amsthm}
\usepackage{color}
\usepackage{pdfsync}
\usepackage{indentfirst}
\date{today}
\allowdisplaybreaks

\usepackage{hyperref}
\hypersetup{hypertex=true,
	colorlinks=true,
	linkcolor=blue,
	anchorcolor=g,
	citecolor=red}

\usepackage{amsmath}
\usepackage{amsthm}

\newtheorem{remark}{Remark}[section]

\newtheorem{theorem}{Theorem}[section]
\newtheorem{proposition}{Proposition}[section]
\newtheorem{lemma}{Lemma}[section]
\newtheorem{corollary}{Corollary}[section]

\newcommand{\beq}{\begin{equation}}
	\newcommand{\eeq}{\end{equation}}
\newcommand{\ben}{\begin{eqnarray}}
	\newcommand{\een}{\end{eqnarray}}
\newcommand{\beno}{\begin{eqnarray*}}
	\newcommand{\eeno}{\end{eqnarray*}}

\numberwithin{equation}{section}

\begin{document}
	\title[Blow-up suppression via the 2-D Taylor-Couette flow]{Blow-up suppression of the
		Patlak-Keller-Segel-Navier-Stokes system via Taylor-Couette flow}
	\author{Shikun~Cui}
	\address[Shikun~Cui]{School of Mathematical Sciences, Dalian University of Technology, Dalian, 116024,  China}
	\email{cskmath@163.com}
	\author{Lili~Wang}
	\address[Lili~Wang]{School of Mathematical Sciences, Dalian University of Technology, Dalian, 116024,  China}
	\email{wllmath@163.com}
	\author{Wendong~Wang}
	\address[Wendong~Wang]{School of Mathematical Sciences, Dalian University of Technology, Dalian, 116024,  China}
	\email{wendong@dlut.edu.cn}
	\date{\today}
	\maketitle

	\begin{abstract} 
		Motivated by the use of Taylor-Couette flow 
		in extracorporeal circulation devices [K$\ddot{\rm o}$rfer et al., 2003, 26(4): 331-338], where it leads to an
		accumulation of platelets and plasma proteins in the vortex center and therefore to a decreased probability of contact between platelets and material surfaces  and its  protein adsorption per square unit is significantly lower than laminar flow.
Increased platelet adhesion or protein adsorption on the device surface can induce platelet aggregation or thrombosis, which  is analogous to the ``blow-up phenomenon" in mathematical modeling. Here
we mathematically analyze this stability mechanism and demonstrate that sufficiently strong flow can prevent blow-up from occurring.
 %we verify this stability mechanism in mathematical and show that the blow-up phenomenon will not happen if the flow is sufficiently strong. 
 In details, we investigate the two-dimensional Patlak-Keller-Segel-Navier-Stokes system in an annular domain around a Taylor-Couette flow 
		$U(r,\theta)=A\big(r+\frac{1}{r} \big)(-\sin\theta, \cos\theta)^{T}$ with $(r,\theta)\in[1,R]\times\mathbb{S}^{1}$,
		and prove that 
		the solutions are globally bounded without any smallness restriction on the initial cell mass or velocity  when $A$ is large.
	\end{abstract}
	
	
	\section{Introduction}
	Guillermo et al. in \cite{GG1999} observed experimentally that Taylor-Couette flow employed in the Vortex Flow Plasmapheretic Reactor (VFPR) demonstrates multiple functional benefits that enhance both performance and safety in extracorporeal heparin management. It also minimizes blood cell damage by effectively separating cellular components from immobilized enzyme beads. Most importantly, it enables safe regional heparinization by efficiently removing heparin in the extracorporeal circuit, maintaining a therapeutic anticoagulant level externally while reducing systemic exposure in the patient. In addition, K\"{o}rfer et al. in \cite{K2003} also found that Taylor-Couette flow  in extracorporeal circulation devices can lead to an
	accumulation of platelets and plasma proteins in the vortex center and therefore to a decreased
	probability of contact between platelets and material surfaces. Especially, at shear rates greater than or equal to $
	550 s^{-1}$, laminar flow resulted in a significantly higher platelet drop and PF4 release than Taylor vortex flow.
	Also protein adsorption per square unit was significantly higher for laminar flow.
	
	As a classic fluid dynamic phenomenon, Taylor-Couette flow describes the steady-state motion of a viscous fluid confined between two coaxial rotating cylinders, first systematically studied by Taylor in the 1920s \cite{Taylor1923}. Despite its conceptually simple geometry, the stability and perturbation of this flow have long presented challenging research questions, leading to extensive experimental, theoretical, and numerical investigations \cite{CI1994,F2018,K1967,P2000}. It remains an active field in fluid mechanics, with many aspects still not fully understood. At the biological level, beyond its use in heparin management, Taylor-Couette flow has proven relevant in several key biomedical applications, including enhancing red blood cell oxygenation \cite{MH1985}, improving plasma filtration efficiency \cite{BJ1989}, and facilitating enzymatic heparin neutralization \cite{AHSL}.

	Inspired by the above important applications of Taylor-Couette flow, 
	consider the following two-dimensional  Patlak-Keller-Segel (PKS) system coupled with the
	Navier-Stokes (NS) equations in a two-dimensional annular region:
	\begin{equation}\label{ini}
		\left\{
		\begin{array}{lr}
			\partial_tn+v\cdot\nabla n=\Delta n-\nabla\cdot(n\nabla c), \\
			\Delta c+n-c=0, \\
			\partial_tv+ v\cdot\nabla v+\nabla P=\Delta v+n\nabla\Phi,\quad\nabla\cdot v=0, \\
			(n,v)\big|_{t=0}=(n_{\rm in},v_{\rm in}),
		\end{array}
		\right.
	\end{equation}
	where $(x,y)\in\mathcal{D}$ and $\mathcal{D}\subset\mathbb{R}^2$ is an annular region.
	Here, $n$ is the cell density, $c$ denotes the concentration of chemoattractant, and $v$ denotes the velocity of fluid. In addition, $P$ is the pressure and $\Phi$ represents the given potential function. Assume that $\Phi=\sqrt{x^2+y^2}$ for simplicity.
	
	When the fluid velocity and the coupling are absent (i.e., $v=0$ and $\Phi=0$), the system \eqref{ini} reduces to the classical  Patlak-Keller-Segel model, which was originally introduced by Patlak \cite{P1953} and further developed by Keller and Segel \cite{KS1970}. The Patlak-Keller-Segel system is commonly used to describe the chemotaxis of microorganisms or cells in response to chemical signals. This fundamental process underlies critical biological behaviors such as nutrient foraging, signal relay, and avoidance of detrimental environments \cite{HT1, HP1}. 
	Up to now, there are many developments for the PKS system on blow-up or the critical mass threshold, and we review some progress briefly.
	In the one-dimensional space, all solutions to the PKS system are globally well-posed \cite{OY2001}. 
	In two-dimensional space, the PKS system, in both its parabolic-elliptic and parabolic-parabolic forms, exhibits a $8\pi$ critical mass.
	Define the initial mass $M:=\|n_{\rm in}\|_{L^{1}}$, and
	if $M< 8\pi$, the solutions of the PKS system are globally well-posed. 
	For the parabolic-elliptic case,
	Wei \cite{W2018} proved that the solution is globally well-posed if and only if $M\leq 8\pi$ (see also \cite{BCM2008}). While the cell mass  $M>8\pi$, the solutions of the PKS system will blow up in finite time, and we refer to  Collot-Ghoul-Masmoudi-Nguyen \cite{CGMN2022}, and  Schweyer \cite{Schweyer1} and the references therein.

	It is a more realistic scenario that chemotactic processes take place in a moving fluid.
	As said in  \cite{KX2015}:
	``{\it A natural question is whether the presence of fluid flow
can affect singularity formation by mixing the bacteria thus making concentration
harder to achieve.}" 
	Kiselev-Xu \cite{KX2015} demonstrated this for stationary relaxation
	enhancing flows and time-dependent Yao-Zlatos near-optimal mixing flows in $\mathbb{T}^d~(d=2,3)$; Bedrossian-He \cite{Bedro2} for non-degenerate shear flows in $\mathbb{T}^2$; and He \cite{he0} for monotone shear flows in $\mathbb{T}\times\mathbb{R}$. For the fully coupled Patlak-Keller-Segel-Navier-Stokes (PKS-NS) system, global regularity for strong Couette flow was proven by Zeng-Zhang-Zi \cite{zeng} in $\mathbb{T}\times\mathbb{R}$. Furthermore, Li-Xiang-Xu \cite{Li0} utilized Poiseuille flow, while Cui-Wang \cite{cui1} considered Navier-slip boundary conditions in $\mathbb{T}\times \mathbb{I}$. Recently, Chen-Wang-Yang investigated the suppression of blow-up in solutions to the Patlak-Keller-Segel (-Navier-Stokes) system by a large Couette flow and  established a precise relationship between the amplitude of the Couette flow and the initial data \cite{CWY2025}.
	More references on higher dimensional cases or other  methods to suppress blow-up, we refer to \cite{TW2016, Hu0, Hu1, he24-1,he24-2,CWW1,CWWT3,CWWTIT} and the references therein.
	
 In the plane coordinate, to deal with the pressure $P$, it is common to introduce the vorticity $\omega$  and the stream function $\phi$ satisfying
	$\omega=\partial_{x}v_{2}-\partial_{y}v_{1}$ and $v=(-\partial_{y}\phi, \partial_{x}\phi)^T$.
	When considering the radial vorticity $\omega(x,y)=\omega(r)$ and stream function $\phi(x,y)=\phi(r)$ with $r=\sqrt{x^{2}+y^{2}}$, the vorticity and the velocity field are reduced to
	\begin{equation}\label{radial}
		\left\{
		\begin{array}{lr}
			\omega(x,y)=\omega(r)=\Delta\phi=\phi''(r)+\frac{1}{r}\phi'(r), \\
			v(x,y)=\left(
			\begin{array}{c}
				-\partial_{y}\phi\\
				\partial_{x}\phi\\
			\end{array}
			\right)=\left(
			\begin{array}{c}
				-\sin\theta\\
				\cos\theta\\
			\end{array}
			\right)\phi'(r).
		\end{array}
		\right.
	\end{equation}
	When vorticity $\omega=const$, 
	the stream function $\phi$ defined by $(\ref{radial})_{1}$ indicates
	\begin{equation}\label{eq:phi}
		\phi''(r)+\frac{1}{r}\phi'(r)=const.
	\end{equation}
	In the polar coordinate, the functions $v(x,y)$ and $\omega(x,y)$ are denoted as $U(r, \theta)$ and $\Omega(r)$, respectively. Solving (\ref{eq:phi})  yields their expressions
	\begin{equation}\label{TC flow}
		U(r,\theta)=
		\left(
		\begin{array}{c}
			U_{1}\\
			U_{2}\\
		\end{array}
		\right)=\left(
		\begin{array}{c}
			-\sin\theta\\
			\cos\theta\\
		\end{array}
		\right)\left(Ar+\frac{B}{r}\right),\quad \Omega(r)=2A,
	\end{equation}
	where $A, B$ are constants and spatial variables $(r,\theta)$ belong to a domain $\mathcal{D}=[1,R]\times\mathbb{S}^{1}$. 
	The velocity field $U(r,\theta)$ given in (\ref{TC flow}) is called as Taylor-Couette (TC) flow,
	which is a steady-state solution of 2D incompressible NS equations. In the meanwhile, $\{n,c,v\}=\{0,0,U(r,\theta)\}$  is also a steady-state solution of the PKS-NS system (\ref{ini}).

	Next, we focus on 
	the  blow-up suppression for the PKS-NS system via Taylor-Couette flow in an annulus. 
	Introduce a perturbation around the two-dimensional TC flow $U(r,\theta)$ from (\ref{TC flow}) for the case $A=B$. Setting  $w=\omega-\Omega, u=v-U$, with $\varphi$ being the stream function satisfying $\Delta\varphi=w$ and $u=(-\partial_{y}\varphi, \partial_{x}\varphi)$. After the time rescaling $t\mapsto\frac{t}{A}$, we rewrite the system (\ref{ini}) in polar coordinates:
	\begin{equation}\label{ini1}
		\left\{
		\begin{array}{lr}
			\partial_{t}n-\frac{1}{A}(\partial_{r}^{2}+\frac{1}{r}\partial_{r}+\frac{1}{r^{2}}\partial_{\theta}^{2})n+(1+\frac{1}{r^{2}})\partial_{\theta}n+\frac{1}{Ar}(\partial_{r}\varphi\partial_{\theta}n-\partial_{\theta}\varphi\partial_{r}n)\\\qquad\qquad\qquad\qquad\qquad\qquad
			=-\frac{1}{Ar}\partial_{r}(rn\partial_{r}c)-\frac{1}{Ar^{2}}\partial_{\theta}(n\partial_{\theta}c), \\
			(\partial_{r}^{2}+\frac{1}{r}\partial_{r}+\frac{1}{r^{2}}\partial_{\theta}^{2})c+n-c=0,\\
			\partial_{t}w-\frac{1}{A}(\partial_{r}^{2}+\frac{1}{r}\partial_{r}+\frac{1}{r^{2}}\partial_{\theta}^{2})w+(1+\frac{1}{r^{2}})\partial_{\theta}w+\frac{1}{Ar}(\partial_{r}\varphi\partial_{\theta}w-\partial_{\theta}\varphi\partial_{r}w)=-\frac{1}{Ar}\partial_{\theta}n, \\
			(\partial_{r}^{2}+\frac{1}{r}\partial_{r}+\frac{1}{r^{2}}\partial_{\theta}^{2})\varphi=w,
		\end{array}
		\right.
	\end{equation}
	together with the Dirichlet boundary conditions
	\begin{equation}\label{boundary condition}
		\begin{aligned}
			n|_{r=1,R}=0,\quad c|_{r=1,R}=0,\quad w|_{r=1,R}=0,\quad \varphi|_{r=1,R}=0
		\end{aligned}
	\end{equation}
	with $(r,\theta)\in\mathcal{D}=[1,R]\times\mathbb{S}^{1}$ and $t\geq 0$.
	
	Our main result is stated as follows.
	\begin{theorem}\label{thm:main}
		Assume that the initial data $0\leq n_{\rm in}\in L^{\infty}\cap H^{1}(\mathcal{D})$ and $u_{\rm in}\in H^{2}(\mathcal{D})$. There exists a positive $A_{1}$
		depending on $\|n_{\rm in}\|_{L^{\infty}\cap H^{1}(\mathcal{D})}$ and $\|u_{\rm in}\|_{H^{2}(\mathcal{D})}$, such
		that if $A\geq A_{1}$, then the solutions of \eqref{ini1}-\eqref{boundary condition} are globally bounded and satisfy the follow stability estimates:
		\begin{itemize}
			\item[(i)]  Uniform bounded-ness estimates: 
			\begin{equation*}
				\begin{aligned}
					&\|u\|_{L^{\infty}L^{\infty}}\leq C(\|n_{\rm in}\|_{H^{1}(\mathcal{D})}, \|u_{\rm in}\|_{H^{2}(\mathcal{D})},R),\\
					&\|n\|_{L^{\infty}L^{\infty}}\leq C(\|n_{\rm in}\|_{L^{\infty}\cap H^{1}(\mathcal{D})}, \|u_{\rm in}\|_{H^{2}(\mathcal{D})},R).
				\end{aligned}
			\end{equation*}
			\item[(ii)] Enhanced dissipation estimates:
			\begin{equation*}
				\begin{aligned}
					&\Big\|{\rm e}^{aA^{-\frac13}|\partial_\theta|^{\frac23}R^{-2}t} \Big(n-\frac{1}{2\pi}\int_{0}^{2\pi}n d\theta\Big)
					\Big\|_{L^{2}}\leq C(R)\|n_{\rm in}\|_{H^{1}(\mathcal{D})},\\
					&\Big\|{\rm e}^{aA^{-\frac13}|\partial_\theta|^{\frac23}R^{-2}t} \Big(w-\frac{1}{2\pi}\int_{0}^{2\pi}w d\theta\Big)
					\Big\|_{L^{2}}\leq C(R)\|u_{\rm in}\|_{H^{2}(\mathcal{D})}.
				\end{aligned}
			\end{equation*}
		\end{itemize}
	\end{theorem}

	\begin{remark}
The Taylor-Couette flow has been successfully implemented in biomedical devices such as the  VFPR, where its unique vortex structure significantly enhances hemocompatibility. It reduces platelet activation and protein adsorption by promoting the accumulation of cellular components in the vortex center, thereby lowering the risk of thrombogenesis and improving the safety of extracorporeal circulation systems \cite{K2003, GG1999}. 
Increased platelet adhesion or protein adsorption on the device surface can induce platelet aggregation or thrombosis, which  poses a huge threat to human life. The above theorem shows  that sufficiently strong flow can prevent the aggregation or blow-up from occurring. As shown in \cite{K2003} at shear rates $ G\geq
	550 s^{-1}$, laminar flow resulted in a significantly higher platelet drop and PF4 release than Taylor vortex flow. Here $G$ is similar as $A$. In fact, let $w_1$ and $w_2$ denote the angular velocity of the
inner or outer cylinder, where  $w_1r_1=v_\theta|_{r_1}=A(r+\frac1r)|_{r_1}$ and $w_2r_2=v_\theta|_{r_2}=A(r+\frac1r)|_{r_2}$. Then
\beno
G=\frac{2(r_1^2w_1+r_2^2w_2)}{r_2^2-r_1^2},\eeno
which implies $G=2A\frac{3+R^2}{R^2-1}$ when $r_1=1, r_2=R.$ It is interesting to estimate the value of $A$ in mathematics, which will be investigated in our future work.
	\end{remark}
	
	\begin{remark}
		Compared with previous results on planar flows (e.g., Couette or Poiseuille flows) usually set in Cartesian coordinates, to our best knowledge, the above result gives the first rigorous proof of global regularity for the Patlak-Keller-Segel-Navier-Stokes system driven by a non-planar shear flow, specifically the Taylor-Couette flow in an annular domain. One of main difficulties lies in the $T_1$ term of the estimates of $n$:
		\begin{equation*}
			\begin{aligned}
				\|n\|_{Y_{a}}
				&\leq C\Big( \sum_{k\neq 0, k\in\mathbb{Z}}\|n_{k}(0)\|_{L^{2}}+\sum_{k\neq 0, k\in\mathbb{Z}}A^{\frac16}|k|^{-\frac13}\|{\rm e}^{aA^{-\frac13}|k|^{\frac23}R^{-2}t}kf_{1}\|_{L^{2}L^{2}}+\cdots\Big)\\
				&=:C\Big(\sum_{k\neq 0, k\in\mathbb{Z}}\|n_{k}(0)\|_{L^{2}}+T_{1}+\cdots \Big).
			\end{aligned}
		\end{equation*}
where
\begin{equation*}
				\begin{aligned}
					\|f_{1}\|_{L^{2}}&\leq \frac{C(R)}{A}\Big(\sum_{l\in\mathbb{Z}\backslash\{0,k\}}|l|^{-\frac12}\|w_{l}\|_{L^{2}}\|n_{k-l}\|_{L^{2}}+\cdots\Big)
				\end{aligned}
			\end{equation*}
(see \eqref{nk Xa} and \eqref{T2 0}). We estimate the norm by considering the characteristics of each of the four cases based on frequency.
Moreover, our result requires no smallness assumption on the initial cell mass or on the initial velocity field; the global bounded-ness is achieved solely by the strength of the Taylor-Couette flow (i.e., a sufficiently large $A$).
	\end{remark}

\begin{remark}\label{local well-posedness}
		The result of local well-posedness of the system (\ref{ini1}) is standard, which can be refered to \cite{Hu1, Wink2013}, and we omitted it.
	\end{remark}	
	The stabilizing phenomenon is fundamentally caused by the enhanced dissipation induced by the Taylor-Couette flow. We first recall the space-time estimate of the following system (see Proposition 6.1 in \cite{AHL2024}), which plays a crucial role in the subsequent analysis. Let
	\begin{equation}\label{eq:space time}
		\left\{
		\begin{array}{lr}
			\partial_{t}h-\frac{1}{A}\Big(\partial_{r}^{2}-\frac{k^{2}-\frac14}{r^{2}} \Big)h+\frac{ik}{r^{2}}h+\frac{1}{r}\big[ikh_{1}-r^{\frac12}\partial_{r}(r^{\frac12}h_{2}) \big]=0,\\
			h|_{t=0}=h(0),\quad h|_{r=1,R}=0,
		\end{array}
		\right.
	\end{equation}
	where $h_{1}$ and $ h_{2}$ are given functions.
	\begin{proposition}\label{prop:space time}
		For $k\in\mathbb{Z}\backslash\{0\}$, let $h$ be a solution to \eqref{eq:space time} with $h(0)\in L^{2}$. Given $\log R\leq CA^{\frac13}$, then there exists a constant $a>0$ independent of $A, k, R$, such that it holds
		\begin{equation*}
			\begin{aligned}
				&\|{\rm e}^{aA^{-\frac13}|k|^{\frac23}R^{-2}t}h\|_{L^{\infty}L^{2}}+A^{-\frac16}|k|^{\frac13}R^{-1}\|{\rm e}^{aA^{-\frac13}|k|^{\frac23}R^{-2}t}h\|_{L^{2}L^{2}}\\&+A^{-\frac12}\|{\rm e}^{aA^{-\frac13}|k|^{\frac23}R^{-2}t}\partial_{r}h\|_{L^{2}L^{2}}+A^{-\frac12}|k|\Big\|{\rm e}^{aA^{-\frac13}|k|^{\frac23}R^{-2}t}\frac{h}{r}\Big\|_{L^{2}L^{2}}\\\leq&C\left(\|h(0)\|_{L^{2}}+A^{\frac16}|k|^{-\frac13}\|{\rm e}^{aA^{-\frac13}|k|^{\frac23}R^{-2}t}kh_{1}\|_{L^{2}L^{2}}+A^{\frac12}\|{\rm e}^{aA^{-\frac13}|k|^{\frac23}R^{-2}t}h_{2}\|_{L^{2}L^{2}} \right).
			\end{aligned}
		\end{equation*}
	\end{proposition}

	Here are some notations used in this paper.
	
	\noindent\textbf{Notations}:
	\begin{itemize}
		
		\item The Fourier transform is defined by
		\begin{equation}
			f(t,r,\theta)=\sum_{k\in\mathbb{Z}}\widehat{f}_{k}(t,r){\rm e}^{ik\theta}, \nonumber
		\end{equation}
		where $\widehat{f}_{k}(t,r)=\frac{1}{2\pi}\int_{0}^{2\pi}f(t,r,\theta){\rm e}^{-ik\theta}d\theta$.

		\item For a given function $f=f(t,r,\theta)$,  we write its  zero mode by
		$$P_0f=\widehat{f}_0=\frac{1}{2\pi}\int_{0}^{2\pi}f(t,r,\theta)d\theta.$$ 
		
		\item  For  given functions $f=f(t,r,\theta)$ and $g=g(t,r)$, 
		their space norm  and  time-space norm  are defined as 
		$$\|f\|_{L^{p}([1,R]\times\mathbb{S}^{1})}=\left(\int_{0}^{2\pi}\int_{1}^{R}|f|^{p}drd\theta \right)^{\frac{1}{p}},
		\quad \|g\|_{L^{p}([1,R])}=\left(\int_{1}^{R}|g|^{p}dr \right)^{\frac{1}{p}}$$ 
		and 
		$$\|f\|_{L^{q}L^{p}}=\left\|\|f\|_{L^{p}([1,R]\times\mathbb{S}^{1})} \right\|_{L^{q}(0,t)},
		\quad \|g\|_{L^{q}L^{p}}=\left\|\|g\|_{L^{p}([1,R])} \right\|_{L^{q}(0,t)}.$$ Moreover, $\langle \cdot, \cdot\rangle$ denotes the standard $L^{2}$ scalar product.
		
		\item The total mass $\|n_{\rm in}\|_{L^{1}}$ is denoted by $M$. Clearly, Green’s identity gives
		\begin{equation*}
			\begin{aligned}
				&	\|n(t)\|_{L^{1}}\leq\|n_{\rm in}\|_{L^{1}}=:M.
			\end{aligned}
		\end{equation*}
		
		\item Throughout this paper, we denote by $ C $ a positive constant independent of $A$, $t$ and the initial data, and it may be different from line to line. $C(R)$  denotes a constant depending on the parameter $R$.
	\end{itemize}
	
	The rest part of this paper is organized as follows. In Section \ref{ideas}, some key ideas and the proof of
	Theorem \ref{thm:main} are presented. Section \ref{priori estimate} is devoted to providing a priori estimates and zero mode estimates, which are essential for the subsequent analysis.  The energy estimates for $E(t)$ and the proof of Proposition \ref{prop:E(t)} are established in Section \ref{estimate E(t)}. In Section \ref{estimate infty}, we complete the proof of Proposition \ref{prop:n infty}.

	\section{Sketch of the proof of Theorem \ref{thm:main}}\label{ideas}
	In this section, we present some
	key ideas and the proof of Theorem \ref{thm:main}.
	
	Note that the coordinate $\theta$ in TC flow is defined on $\mathbb{S}^{1}$, and it is natural to applying Fourier transform on the $\theta$ direction. Then taking Fourier transform for (\ref{ini1})-(\ref{boundary condition}) with respect to $\theta$, we obtain
	\begin{equation}\label{eq:fourier}
		\left\{
		\begin{array}{lr}
			\partial_{t}\widehat{n}_{k}-\frac{1}{A}(\partial_{r}^{2}+\frac{1}{r}\partial_{r}-\frac{k^{2}}{r^{2}})\widehat{n}_{k}+(1+\frac{1}{r^{2}})ik\widehat{n}_{k}+\frac{1}{Ar}\sum_{l\in\mathbb{Z}}i(k-l)\partial_{r}\widehat{\varphi}_{l}\widehat{n}_{k-l}\\-\frac{1}{Ar}\sum_{l\in\mathbb{Z}}il\widehat{\varphi}_{l}\partial_{r}\widehat{n}_{k-l}=-\frac{1}{Ar}\sum_{l\in\mathbb{Z}}\partial_{r}(r\widehat{n}_{l}\partial_{r}\widehat{c}_{k-l})-\frac{ik}{Ar^{2}}\sum_{l\in\mathbb{Z}}i(k-l)\widehat{n}_{l}\widehat{c}_{k-l},\\
			(\partial_{r}^{2}+\frac{1}{r}\partial_{r}-\frac{k^{2}}{r^{2}})\widehat{c}_{k}+\widehat{n}_{k}-\widehat{c}_{k}=0,\\
			\partial_{t}\widehat{w}_{k}-\frac{1}{A}(\partial_{r}^{2}+\frac{1}{r}\partial_{r}-\frac{k^{2}}{r^{2}})\widehat{w}_{k}+(1+\frac{1}{r^{2}})ik\widehat{w}_{k}+\frac{1}{Ar}\sum_{l\in\mathbb{Z}}i(k-l)\partial_{r}\widehat{\varphi}_{l}\widehat{w}_{k-l}\\-\frac{1}{Ar}\sum_{l\in\mathbb{Z}}il\widehat{\varphi}_{l}\partial_{r}\widehat{w}_{k-l}=-\frac{ik}{Ar}\widehat{n}_{k},\\
			(\partial_{r}^{2}+\frac{1}{r}\partial_{r}-\frac{k^{2}}{r^{2}})\widehat{\varphi}_{k}=\widehat{w}_{k},\\
			\widehat{n}_{k}|_{r=1,R}=\widehat{c}_{k}|_{r=1,R}=\widehat{w}_{k}|_{r=1,R}=\widehat{\varphi}_{k}|_{r=1,R}=0.
			
		\end{array}
		\right.
	\end{equation}
	Inspired by  An-He-Li \cite{AHL2024}, we introduce the weight $r^{\frac12}$ to eliminate the derivative $\frac{1}{r}\partial_{r}$. Specifically,  define 
	\begin{equation*}
		n_{k}:=r^{\frac12}{\rm e}^{ikt}\widehat{n}_{k},\quad c_{k}:=r^{\frac12}{\rm e}^{ikt}\widehat{c}_{k},\quad 
		w_{k}:=r^{\frac12}{\rm e}^{ikt}\widehat{w}_{k},\quad 
		\varphi_{k}:=r^{\frac12}{\rm e}^{ikt}\widehat{\varphi}_{k}.
	\end{equation*}
	It follows that
	\begin{equation}\label{define nk ck wk}
		\|F_{k}\|_{L^{p}}\leq R^{\frac12}\|\widehat{F}_{k}\|_{L^{p}},\quad {\rm for}\quad F\in\{n, c, w, \varphi \}~{\rm and}~p\in\{2, \infty\}.
	\end{equation}
	Denote the operator $\mathcal{L}_{k}$ as
	\begin{equation}\label{define:Lk}
		\mathcal{L}_{k}f:=-\frac{1}{A}\Big(\partial_{r}^{2}-\frac{k^{2}-\frac14}{r^{2}} \Big)f+\frac{ik}{r^{2}}f.
	\end{equation}
	Thus, the system (\ref{eq:fourier})  is transformed into
	\begin{equation}\label{eq:fourier 1}
		\left\{
		\begin{array}{lr}
			\partial_{t}n_{k}+\mathcal{L}_{k}n_{k}+\frac{\big[ik\sum_{l\in\mathbb{Z}}\partial_{r}(r^{-\frac12}\varphi_{l})n_{k-l}-r^{\frac12}\partial_{r}\left(\sum_{l\in\mathbb{Z}}ilr^{-1}\varphi_{l}n_{k-l} \right) \big]}{Ar}
			\\=-\frac{1}{Ar^{\frac12}}\partial_{r}\big[\sum_{l\in\mathbb{Z}}r^{\frac12}n_{l}\partial_{r}(r^{-\frac12}c_{k-l}) \big]-\frac{ik}{Ar^{\frac52}}\sum_{l\in\mathbb{Z}}i(k-l)n_{l}c_{k-l},\\
			\big(\partial_{r}^{2}-\frac{k^{2}-\frac14}{r^{2}} \big)c_{k}+n_{k}-c_{k}=0,\\
			\partial_{t}w_{k}+\mathcal{L}_{k}w_{k}
			+\frac{\big[ik\sum_{l\in\mathbb{Z}}\partial_{r}(r^{-\frac12}\varphi_{l})w_{k-l}-r^{\frac12}\partial_{r}\big(\sum_{l\in\mathbb{Z}}ilr^{-1}\varphi_{l}w_{k-l} \big) \big]}{Ar}=-\frac{ik}{Ar}n_{k},\\
			\big(\partial_{r}^{2}-\frac{k^{2}-\frac14}{r^{2}} \big)\varphi_{k}=w_{k},\\
			n_{k}|_{r=1,R}=c_{k}|_{r=1,R}=w_{k}|_{r=1,R}=\varphi_{k}|_{r=1,R}=0.
		\end{array}
		\right.
	\end{equation}
	
	We introduce the following norms
	\begin{equation}\label{X_a^k}
		\begin{aligned}
			\|f_k\|_{X_{a}^{k}}&=\|{\rm e}^{aA^{-\frac13}|k|^{\frac23}R^{-2}t}f_k\|_{L^{\infty}L^{2}}+A^{-\frac16}|k|^{\frac13}R^{-1}\|{\rm e}^{aA^{-\frac13}|k|^{\frac23}R^{-2}t}f_k\|_{L^{2}L^{2}}\\&\quad +A^{-\frac12}\|{\rm e}^{aA^{-\frac13}|k|^{\frac23}R^{-2}t}\partial_{r}f_k\|_{L^{2}L^{2}}
			+A^{-\frac12}|k|\Big\|{\rm e}^{aA^{-\frac13}|k|^{\frac23}R^{-2}t}\frac{f_k}{r}\Big\|_{L^{2}L^{2}}
		\end{aligned}
	\end{equation}
	and 
	\begin{equation}\label{Xa}
		\|f\|_{Y_{a}}=\sum_{k\neq 0, k\in\mathbb{Z}}\|f_k\|_{X_{a}^{k}}.
	\end{equation}
	Moreover, we construct the energy functional as follows:
	\begin{equation*}
		E(t)=\|n\|_{Y_{a}}+\|w\|_{Y_{a}}
	\end{equation*}
	with the initial norm    
	\begin{equation*}
		E_{\rm in}=\sum_{k\neq 0, k\in\mathbb{Z}}\|n_{k}(0)\|_{L^{2}}+\sum_{k\neq 0, k\in\mathbb{Z}}\|w_{k}(0)\|_{L^{2}}.
	\end{equation*}
	The proof of the main result relies on a bootstrap argument. Let's designate $T$ as the terminal point of the largest range $[0,T]$ such that the following hypothesis hold
	\begin{equation}\label{assumption}
		\begin{aligned}
			E(t) &\leq 2\mathcal{Q}_{1},\\
			\|n\|_{L^{\infty}L^{\infty}}&\leq 2\mathcal{Q}_{2}
		\end{aligned}
	\end{equation}
	for $t\in[0,T]$, where $\mathcal{Q}_{1}$ and $\mathcal{Q}_{2}$ are constants independent of $t$ and $A$ and will be decided during the calculation. 
	
	The following propositions are key to obtaining the main results.
	\begin{proposition}\label{prop:E(t)}
		Assume that the initial data $0\leq n_{\rm in}\in L^{\infty}\cap H^{1}(\mathcal{D})$ and $u_{\rm in}\in H^{2}(\mathcal{D})$. Under the conditions of 
		\eqref{assumption}, there exists a positive constant $\mathcal{C}_{2}$ depending on $\|n_{\rm in}\|_{L^{\infty}\cap H^{1}(\mathcal{D})}$
		and $\|u_{\rm in}\|_{H^{2}(\mathcal{D})}$, such that if $A\geq \mathcal{C}_{2}$, then there holds
		\begin{equation*}
			E(t)\leq \mathcal{Q}_{1}
		\end{equation*}
		for all $t\in[0,T]$.
	\end{proposition}
	\begin{proposition}\label{prop:n infty}
		Under the assumptions of Proposition \ref{prop:E(t)}, there exists a positive constant $\mathcal{Q}_{2}$ depending on $\|u_{\rm in}\|_{H^{2}(\mathcal{D})}$ and $\|n_{\rm in}\|_{L^{\infty}\cap H^{1}(\mathcal{D})}$, such that 
		\begin{equation*}
			\|n\|_{L^{\infty}L^{\infty}}\leq\mathcal{Q}_{2}
		\end{equation*}
		for all $t\in[0,T]$.
	\end{proposition}
	
	\begin{proof}[Proof of Theorem \ref{thm:main}]
		Taking $A_{1}=\max\{\mathcal{C}_{1}, \mathcal{C}_{2}\}$ and	combining Proposition \ref{prop:E(t)} and Proposition \ref{prop:n infty} with
		Corollary \ref{coro1}
		and the well-posedness of system as in Remark \ref{local well-posedness}, we complete the proof.
	\end{proof}

	\section{A priori estimates and zero mode estimates}\label{priori estimate}
	\subsection{Elliptic estimates for $c$}
	\begin{lemma}\label{lem:ck}
		Suppose that $|k|\geq 1$. Let
		\begin{equation}\label{eq:ck}
			\Big(\partial_{r}^{2}-\frac{k^{2}-\frac14}{r^{2}} \Big)c_{k}+n_{k}-c_{k}=0,\quad c_{k}|_{r=1,R}=0.
		\end{equation}
		Then it holds that
		\begin{equation*}
			\begin{aligned}
				&	\|\partial_{r}c_{k}\|_{L^{2}}+k\big\|\frac{c_{k}}{r} \big\|_{L^{2}}+\|c_{k}\|_{L^{2}}\leq C\|n_{k}\|_{L^{2}},\\
				&\|r^{2}\partial_{r}^{2}c_{k}\|_{L^{2}}+k^{2}\|c_{k}\|_{L^{2}}+k\|r\partial_{r}c_{k}\|_{L^{2}}\leq C(R) \|n_{k}\|_{L^{2}}
			\end{aligned}
		\end{equation*}
		and 
		\begin{equation*}
			\|c_{k}\|_{L^{\infty}}\leq C(R) \|n_{k}\|_{L^{2}}.
		\end{equation*}
	\end{lemma}
	\begin{proof}
		Multiplying (\ref{eq:ck}) by $-c_{k}$, the energy estimate shows that
		\begin{equation*}
			\|\partial_{r}c_{k}\|_{L^{2}}^{2}+\Big(k^{2}-\frac14 \Big)\big\|\frac{c_{k}}{r} \big\|_{L^{2}}^{2}+\|c_{k}\|_{L^{2}}^{2}=\langle n_{k}, c_{k} \rangle\leq \|n_{k}\|_{L^{2}}\|c_{k}\|_{L^{2}}.
		\end{equation*}
		This gives that
		\begin{equation}\label{ck kck}
			4\|\partial_{r}c_{k}\|_{L^{2}}^{2}+4k^{2}\big\|\frac{c_{k}}{r} \big\|_{L^{2}}^{2}+\|c_{k}\|_{L^{2}}^{2}\leq 2\|n_{k}\|_{L^{2}}^{2}.
		\end{equation}
		Due to (\ref{eq:ck}), there holds
		\begin{equation}\label{ck''}
			\begin{aligned}
				&\quad\|r^{2}\partial_{r}^{2}c_{k}\|_{L^{2}}^{2}+k^{4}\|c_{k}\|_{L^{2}}^{2}-2k^{2}\langle r^{2}\partial_{r}^{2}c_{k}, c_{k}\rangle=\|r^{2}\partial_{r}^{2}c_{k}-k^{2}c_{k}\|_{L^{2}}^{2}\\&=\Big\|\Big(r^{2}-\frac14 \Big)c_{k}-r^{2}n_{k} \Big\|_{L^{2}}^{2}\leq C\left(\|r^{2}c_{k}\|_{L^{2}}^{2}+\|r^{2}n_{k}\|_{L^{2}}^{2} \right).
			\end{aligned}
		\end{equation}
		Using integration by parts, $-2k^{2}\langle r^{2}\partial_{r}^{2}c_{k}, c_{k}\rangle$ can be controlled as
		\begin{equation*}
			\begin{aligned}
				2k^{2}\langle r^{2}\partial_{r}c_{k},\partial_{r}c_{k} \rangle+2k^{2}\langle 2r\partial_{r}c_{k}, c_{k}\rangle=2k^{2}\|r\partial_{r}c_{k}\|_{L^{2}}^{2}-2k^{2}\|c_{k}\|_{L^{2}}^{2}.
			\end{aligned}
		\end{equation*}
		This along with (\ref{ck kck}) and (\ref{ck''}) implies that
		\begin{equation}\label{ck k2}
			\begin{aligned}
				&\quad\|r^{2}\partial_{r}^{2}c_{k}\|_{L^{2}}^{2}+k^{4}\|c_{k}\|_{L^{2}}^{2}+2k^{2}\|r\partial_{r}c_{k}\|_{L^{2}}^{2}\\&\leq C\left(k^{2}\|c_{k}\|_{L^{2}}^{2}+\|r^{2}c_{k}\|_{L^{2}}^{2}+\|r^{2}n_{k}\|_{L^{2}}^{2} \right) \leq C R^{4}\|n_{k}\|_{L^{2}}^{2}.
			\end{aligned}
		\end{equation}
		By applying (\ref{ck kck}) and Gagliardo-Nirenberg inequality, one deduces
		\begin{equation*}
			\begin{aligned}
				\|c_{k}\|_{L^{\infty}}\leq C(R)\|c_{k}\|_{L^{2}}^{\frac12}\|\partial_{r}c_{k}\|_{L^{2}}^{\frac12}\leq C(R) \|n_{k}\|_{L^{2}}.
			\end{aligned}
		\end{equation*}
		By combining  (\ref{ck kck}) and (\ref{ck k2}), the proof is complete.
	\end{proof}
	
	\begin{lemma}\label{lem:c0}
		Let $\widehat{c}_{0}$ and $\widehat{n}_{0}$ be the zero mode of $c$ and $n$, respectively, satisfying 
		\begin{equation}\label{eq:c0}
			-\Big(\partial_{r}^{2}+\frac{1}{r}\partial_{r}\Big)\widehat{c}_{0}+\widehat{c}_{0}=\widehat{n}_{0},\quad \widehat{c}_{0}|_{r=1,R}=0.
		\end{equation}
		Then it holds that
		\begin{equation*}
			\begin{aligned}
				\|r^{\frac12}(\widehat{c}_{0}, \partial_{r}\widehat{c}_{0}, \partial_{r}^{2}\widehat{c}_{0})\|_{L^{2}}&\leq C\|r^{\frac12}\widehat{n}_{0}\|_{L^{2}},\\
				\|(1, \partial_{r} )\widehat{c}_{0}\|_{L^{\infty}}&\leq C(R)\|r^{\frac12}\widehat{n}_{0}\|_{L^{2}}.
			\end{aligned}
		\end{equation*}
	\end{lemma}
	\begin{proof}
		The basic energy estimates yield
		\begin{equation*}
			\|r^{\frac12}\widehat{c}_{0}\|_{L^{2}}^{2}+\|r^{\frac12}\partial_{r}\widehat{c}_{0}\|_{L^{2}}^{2}\leq \|r^{\frac12}\widehat{c}_{0}\|_{L^{2}}\|r^{\frac12}\widehat{n}_{0}\|_{L^{2}},
		\end{equation*} 
		which implies that
		\begin{equation}\label{c0 c0'}
			\|r^{\frac12}\widehat{c}_{0}\|_{L^{2}}^{2}+2\|r^{\frac12}\partial_{r}\widehat{c}_{0}\|_{L^{2}}^{2}\leq\|r^{\frac12}\widehat{n}_{0}\|_{L^{2}}^{2}.
		\end{equation}
		By using (\ref{eq:c0}) and (\ref{c0 c0'}), we get
		\begin{equation*}
			\begin{aligned}
				\|r^{\frac12}\partial_{r}^{2}\widehat{c}_{0}\|_{L^{2}}^{2}\leq C\left(\|r^{-\frac12}\partial_{r}\widehat{c}_{0}\|_{L^{2}}^{2}+\|r^{\frac12}\widehat{c}_{0}\|_{L^{2}}^{2}+\|r^{\frac12}\widehat{n}_{0}\|_{L^{2}}^{2}\right) \leq C\|r^{\frac12}\widehat{n}_{0}\|_{L^{2}}^{2}.
			\end{aligned}
		\end{equation*}
		%where we make use of the fact that $r\in[1,R]$ is bounded.
		Combining this with (\ref{c0 c0'}) and Gagliardo-Nirenberg inequality, one obtains
		$$ \|\widehat{c}_{0}\|_{L^{\infty}}\leq C(R)\|\widehat{c}_{0}\|_{L^{2}}^{\frac12}\|\partial_{r}\widehat{c}_{0}\|_{L^{2}}^{\frac12}\leq C(R)\|r^{\frac12}\widehat{c}_{0}\|_{L^{2}}^{\frac12}\|r^{\frac12}\partial_{r}\widehat{c}_{0}\|_{L^{2}}^{\frac12}\leq C(R)\|r^{\frac12}\widehat{n}_{0}\|_{L^{2}} $$
		and
		\begin{equation*}
			\begin{aligned}
				\|\partial_{r}\widehat{c}_{0}\|_{L^{\infty}}&\leq C(R)\left(\|\partial_{r}\widehat{c}_{0}\|_{L^{2}}^{\frac12}\|\partial_{r}^{2}\widehat{c}_{0}\|_{L^{2}}^{\frac12}+\|\partial_{r}\widehat{c}_{0}\|_{L^{2}} \right)\\
				&\leq C(R)\left(\|r^{\frac12}\partial_{r}\widehat{c}_{0}\|_{L^{2}}^{\frac12}\|r^{\frac12}\partial_{r}^{2}\widehat{c}_{0}\|_{L^{2}}^{\frac12}+\|r^{\frac12}\partial_{r}\widehat{c}_{0}\|_{L^{2}} \right)
				\leq C(R)\|r^{\frac12}\widehat{n}_{0}\|_{L^{2}}.
			\end{aligned}
		\end{equation*}

	\end{proof}

	\subsection{Elliptic estimates for the stream function $\varphi_{k}$ with $k\geq 0$}
	\begin{lemma}[Lemma A.3 in \cite{AHL2024}] \label{lem:varphi k>0}
		Suppose that $|k|\geq 1$. Let $w_{k}=\left(\partial_{r}^{2}-\frac{k^{2}-\frac14}{r^{2}} \right)\varphi_{k}$ with $\varphi_{k}|_{r=1,R}=0$. Then there holds
		\begin{equation*}
			\begin{aligned}
				%			&\|\partial_{r}\varphi_{k}\|_{L^{2}}^{2}+|k|^{2}\left\|\frac{\varphi_{k}}{r}\right\|_{L^{2}}^{2}\lesssim|\langle w_{k}, \varphi_{k}\rangle|\lesssim|k|^{-2}\|rw_{k}\|_{L^{2}}^{2},\\
				%			&\|\partial_{r}\varphi_{k}\|_{L^{2}}^{2}+|k|^{2}\left\|\frac{\varphi_{k}}{r}\right\|_{L^{2}}^{2}\lesssim |k|^{-1}\|r^{\frac12}w_{k}\|_{L^{1}}^{2},\\
				&\|r^{\frac12}\partial_{r}\varphi_{k}\|_{L^{\infty}}+|k|\|r^{-\frac12}\varphi_{k}\|_{L^{\infty}}\leq C(R)|k|^{-\frac12}\|rw_{k}\|_{L^{2}}.
			\end{aligned}
		\end{equation*}
	\end{lemma}
	%\begin{lemma}[Lemma A.4 in \cite{AHL2024}]\label{1D poincare}
	%	For any $f(r)\in L^{2}(1,R)$ with $f|_{r=1,R}=0$ for $R>1$, the following Poincar{$\acute{e}$}-type inequalities hold
	%	\begin{equation*}
		%		\|f\|_{L^{2}}\leq 2\|rf'\|_{L^{2}}\quad{\rm and}\quad\Big\|\frac{f}{r^{\frac12}} \Big\|_{L^{2}}\leq 2\log R\|r^{\frac12}f'\|_{L^{2}}.
		%	\end{equation*}
	%\end{lemma}
	\begin{lemma}\label{lem:varphi k=0}
		Let $\widehat{\varphi}_{0}$ and $\widehat{w}_{0}$ be the zero mode of $\varphi$ and $w$, respectively, satisfying
		\begin{equation}\label{eq:varphi 0}
			\Big(\partial_{r}^{2}+\frac{1}{r}\partial_{r} \Big)\widehat{\varphi}_{0}=\widehat{w}_{0},\quad \widehat{\varphi}_{0}|_{r=1,R}=\widehat{w}_{0}|_{r=1,R}=0.
		\end{equation}
		Then it holds that
		\begin{equation}\label{varphi_0 infty}
			\begin{aligned}
				&\|\widehat{\varphi}_{0}\|_{L^{\infty}}\leq C(R) \|r\widehat{w}_{0}\|_{L^{2}},
				\\&\|\partial_{r}\widehat{\varphi}_{0}\|_{L^{\infty}}\leq C(R)\|r^{\frac32}\widehat{w}_{0}\|_{L^{2}}.
			\end{aligned}
		\end{equation}
	\end{lemma}
	\begin{proof}
		Due to integration by parts, there holds
		\begin{equation*}
			\|\widehat{\varphi}_{0}\|_{L^{2}}^{2}=\int_{1}^{R}\widehat{\varphi}_{0}^{2}dr=-2\int_{1}^{R}r\widehat{\varphi}_{0}\partial_{r}\widehat{\varphi}_{0}dr\leq 2\|r\partial_{r}\widehat{\varphi}_{0}\|_{L^{2}}\|\widehat{\varphi}_{0}\|_{L^{2}}.
		\end{equation*}
		This implies that
		\begin{equation}\label{var 2}
			\|\widehat{\varphi}_{0}\|_{L^{2}}\leq 2\|r\partial_{r}\widehat{\varphi}_{0}\|_{L^{2}}.
		\end{equation}
		Combining it with Gagliardo-Nirenberg inequality, we get
		\begin{equation}\label{var infty}
			\|\widehat{\varphi}_{0}\|_{L^{\infty}}\leq C(R)\|\widehat{\varphi}_{0}\|_{L^{2}}^{\frac12}\|\partial_{r}\widehat{\varphi}_{0}\|_{L^{2}}^{\frac12}\leq C(R) \|\partial_{r}\widehat{\varphi}_{0}\|_{L^{2}}.
		\end{equation}
		Moreover, by using (\ref{var 2}), the energy estimate of (\ref{eq:varphi 0}) indicates that
		\begin{equation*}
			\begin{aligned}
				\|r^{\frac12}\partial_{r}\widehat{\varphi}_{0}\|_{L^{2}}^{2}&=\langle \widehat{w}_{0}, r\widehat{\varphi}_{0} \rangle\leq\|r\widehat{w}_{0}\|_{L^{2}}\|\widehat{\varphi}_{0}\|_{L^{2}}\\&\leq 
				2\|r\widehat{w}_{0}\|_{L^{2}}\|r\partial_{r}\widehat{\varphi}_{0}\|_{L^{2}}\leq 2R^{\frac12}\|r\widehat{w}_{0}\|_{L^{2}}\|r^{\frac12}\partial_{r}\widehat{\varphi}_{0}\|_{L^{2}}.
			\end{aligned}
		\end{equation*}
		Therefore, we obtain 
		\begin{equation*}
			\|r^{\frac12}\partial_{r}\widehat{\varphi}_{0}\|_{L^{2}}\leq 2R^{\frac12}\|r\widehat{w}_{0}\|_{L^{2}}.
		\end{equation*}
		Substituting it into (\ref{var infty}), we arrive at
		\begin{equation*}
			\|\widehat{\varphi}_{0}\|_{L^{\infty}}\leq C(R)\|r\widehat{w}_{0}\|_{L^{2}},
		\end{equation*}
		which implies $(\ref{varphi_0 infty})_{1}$.
		
		The proof of $(\ref{varphi_0 infty})_{2}$ can be found in Lemma A.5 in \cite{AHL2024}, and we omit it.
	\end{proof}
	
	%	\section{The estimates for zero modes}
	\subsection{The $L^{2}$ estimates for zero modes of density and vorticity}
	\begin{lemma}\label{lem:n0 w0}
		Under the assumptions of Proposition \ref{prop:E(t)}, there exists a positive constant $\mathcal{C}_{1}$ independent of $t$ and $A$, such that if $A\geq \mathcal{C}_{1}$, it holds
		\begin{equation}\label{result D1}
			\begin{aligned}
				\|r^{\frac12}\widehat{n}_{0}\|_{L^{\infty}L^{2}}\leq C(R)\left(\big\|\widehat{(n_{\rm in})}_{0}\big\|_{L^{2}}+M^{2}+1 \right)=:D_{1},
			\end{aligned}
		\end{equation}
		\begin{equation}\label{result D2}
			\|r^{\frac12}\widehat{w}_{0}\|_{L^{\infty}L^{2}}+\frac{1}{A^{\frac12}}\|r^{\frac12}\partial_{r}\widehat{w}_{0}\|_{L^{2}L^{2}}
			\leq  C(R)\left(\big\|\widehat{(w_{\rm in})}_{0}\big\|_{L^{2}}^{2}+1\right)=:D_{2}.
		\end{equation}
	\end{lemma}
	\begin{proof}
		{\bf Estimate (\ref{result D1}).}	Recall that $\widehat{n}_{0}$ satisfies
		\begin{equation*}
			\begin{aligned}
				\partial_{t}\widehat{n}_{0}-\frac{1}{A}\Big(\partial_{r}^{2}+\frac{1}{r}\partial_{r} \Big)\widehat{n}_{0}+\frac{1}{Ar}P_{0}(\partial_{r}\varphi\partial_{\theta}n-\partial_{\theta}\varphi\partial_{r}n)=-\frac{1}{Ar}\partial_{r}\left[P_{0}(rn\partial_{r}c) \right]
			\end{aligned}
		\end{equation*}
		with $\widehat{n}_{0}|_{r=1,R}=0$.
		Multiplying it by $r\widehat{n}_{0}$ and integrating with $r$ over $[1,R]$, we get
		\begin{equation*}
			\begin{aligned}
				&\langle \partial_{t}\widehat{n}_{0}-\frac{1}{A}\big(\partial_{r}^{2}+\frac{1}{r}\partial_{r}\big)\widehat{n}_{0}, r\widehat{n}_{0} \rangle=\langle -\frac{1}{Ar}P_{0}(\partial_{r}\varphi\partial_{\theta}n-\partial_{\theta}\varphi\partial_{r}n)-\frac{1}{Ar}\partial_{r}[P_{0}(rn\partial_{r}c)], r\widehat{n}_{0} \rangle.
			\end{aligned}
		\end{equation*}
		Observing that $\widehat{n}_{0}|_{r=1,R}=0$ and applying integration by parts to the above equation, we have
		\begin{equation}\label{estimate n0}
			\begin{aligned}
				&\frac12\frac{d}{dt}\|r^{\frac12}\widehat{n}_{0}\|_{L^{2}}^{2}+\frac{1}{A}\|r^{\frac12}\partial_{r}\widehat{n}_{0}\|_{L^{2}}^{2}\\=&-\frac{1}{A}\langle P_{0}[\partial_{r}(\varphi\partial_{\theta}n)-\partial_{\theta}(\varphi\partial_{r}n)], \widehat{n}_{0} \rangle+\frac{1}{A}\langle P_{0}(rn\partial_{r}c), \partial_{r}\widehat{n}_{0} \rangle=:I_{1}+I_{2}.
			\end{aligned}
		\end{equation}
		For given functions $f(t,r,\theta)$ and $g(t,r,\theta)$, it follows from  Fourier series that
		\begin{equation*}
			\begin{aligned}
				f(t,r,\theta)=\sum_{k\in\mathbb{Z}}\widehat{f}_{k}(t,r){\rm e}^{ik\theta}, \quad
				g(t,r,\theta)=\sum_{k\in\mathbb{Z}}\widehat{g}_{k}(t,r){\rm e}^{ik\theta},
			\end{aligned}
		\end{equation*}
		Nonlinear interactions between $f$ and $g$ show that  
		\begin{equation}\label{fg zero}
			P_{0}(fg)=\widehat{(fg)}_{0}=\sum_{k\in\mathbb{Z}}\widehat{f}_{k}(t,r)\widehat{g}_{-k}(t,r)
			=\widehat{f}_0\widehat{g}_0+\sum_{k\in\mathbb{Z},k\neq0}\widehat{f}_{k}(t,r)\widehat{g}_{-k}(t,r).
		\end{equation}
		For $I_{1}$, as $\partial_{\theta}\widehat{(\varphi\partial_{r}n)}_{0}=0$, $\partial_{\theta}\widehat{n}_{0}=0$  and (\ref{fg zero}), 
		by H\"{o}lder's inequality, 
		we obtain that 
		\begin{equation*}
			\begin{aligned}
				I_{1}=\frac{1}{A}
				\left\langle \widehat{(\varphi\partial_{\theta}n)}_{0}, \partial_{r}\widehat{n}_{0} \right\rangle 
				&\leq\frac{1}{A}\Big\|\sum_{k\in\mathbb{Z},k\neq0}\widehat{\varphi}_{k}(t,r)\widehat{\partial_\theta n}_{-k}(t,r)\Big\|_{L^2}
				\|r^{\frac12}\partial_{r}\widehat{n}_{0}\|_{L^{2}}\\
				&\leq \frac{1}{4A}\|r^{\frac12}\partial_{r}\widehat{n}_{0}\|_{L^{2}}^{2}+\frac{C}{A}\Big\|\sum_{k\in\mathbb{Z},k\neq0}k\widehat{\varphi}_{k}(t,r)\widehat{n}_{-k}(t,r)\Big\|_{L^2}^2.
			\end{aligned}
		\end{equation*}
		For $I_{2}$, by using (\ref{fg zero}) and Lemma \ref{lem:c0}, we get
		\begin{equation*}
			\begin{aligned}
				I_{2}&=\frac{1}{A}\langle r\widehat{n}_{0}\partial_{r}\widehat{c}_{0}, \partial_{r}\widehat{n}_{0} \rangle
				+\frac{1}{A}\Big\langle 
				r\sum_{k\in\mathbb{Z},k\neq0}\widehat{n}_{k}(t,r)\widehat{\partial_r c}_{-k}(t,r), \partial_{r}\widehat{n}_{0}\Big\rangle\\
				&\leq\frac{1}{A}\|\partial_{r}\widehat{c}_{0}\|_{L^{\infty}}\|r^{\frac12}\widehat{n}_{0}\|_{L^{2}}\|r^{\frac12}\partial_{r}\widehat{n}_{0}\|_{L^{2}}
				+\frac{R^{\frac12}}{A}\Big\|\sum_{k\in\mathbb{Z},k\neq0}\widehat{n}_{k}(t,r)\widehat{\partial_r c}_{-k}(t,r)\Big\|_{L^{2}}\|r^{\frac12}\partial_{r}\widehat{n}_{0}\|_{L^{2}}\\
				&\leq\frac{1}{4A}\|r^{\frac12}\partial_{r}\widehat{n}_{0}\|_{L^{2}}^{2}+\frac{C}{A}\|r^{\frac12}\widehat{n}_{0}\|_{L^{2}}^{4}+\frac{C(R)}{A}\Big\|\sum_{k\in\mathbb{Z},k\neq0}\widehat{n}_{k}(t,r)\widehat{\partial_r c}_{-k}(t,r)\Big\|^2_{L^{2}}.
			\end{aligned}
		\end{equation*}
		Collecting the estimates of $I_{1}$ and $I_{2}$, (\ref{estimate n0}) yields that
		\begin{equation}\label{estimate n0 1}
			\begin{aligned}
				&\frac{d}{dt}\|r^{\frac12}\widehat{n}_{0}\|_{L^{2}}^{2}+\frac{1}{A}\|r^{\frac12}\partial_{r}\widehat{n}_{0}\|_{L^{2}}^{2}
				\leq\frac{C}{A}\|r^{\frac12}\widehat{n}_{0}\|_{L^{2}}^{4}
				\\&	+\frac{C(R)\big(\|\sum_{k\in\mathbb{Z},k\neq0}k\widehat{\varphi}_{k}(t)\widehat{n}_{-k}(t)\|_{L^2}^2
					+\|\sum_{k\in\mathbb{Z},k\neq0}\widehat{n}_{k}(t)\widehat{\partial_r c}_{-k}(t)\|^2_{L^{2}}
					\big)}{A}.
			\end{aligned}
		\end{equation}
		Due to  Gagliardo-Nirenberg inequality, there holds
		\begin{equation*}
			\begin{aligned}
				-\|\partial_{r}\widehat{n}_{0}\|_{L^{2}}^{2}\leq-\frac{\|\widehat{n}_{0}\|_{L^{2}}^{6}}{C\|\widehat{n}_{0}\|_{L^{1}}^{4}}\leq-\frac{\|\widehat{n}_{0}\|_{L^{2}}^{6}}{CM^{4}}.
			\end{aligned}
		\end{equation*}
		As $r\in[1,R]$, we infer from the above inequality that
		\begin{equation}\label{GN}
			-\|r^{\frac12}\partial_{r}\widehat{n}_{0}\|_{L^{2}}^{2}\leq-\frac{\|r^{\frac12}\widehat{n}_{0}\|_{L^{2}}^{6}}{CR^{3}M^{4}}.
		\end{equation}
		For  all $t\geq 0$, we denote $G(t)$ by
		\begin{equation*}
			G(t):=\frac{C(R)}{A}\int_{0}^{t}\Big(\Big\|\sum_{k\in\mathbb{Z},k\neq0}k\widehat{\varphi}_{k}(s)\widehat{n}_{-k}(s)\Big\|_{L^2}^2
			+\Big\|\sum_{k\in\mathbb{Z},k\neq0}\widehat{n}_{k}(s)\widehat{\partial_r c}_{-k}(s)\Big\|^2_{L^{2}}
			\Big)ds.
		\end{equation*}
		Substituting (\ref{GN}) into (\ref{estimate n0 1}), we rewrite (\ref{estimate n0 1}) into
		\begin{equation*}
			\begin{aligned}
				\frac{d}{dt}\left(\|r^{\frac12}\widehat{n}_{0}\|_{L^{2}}^{2}-G(t) \right)\leq&-\frac{\|r^{\frac12}\widehat{n}_{0}\|_{L^{2}}^{4}}{CAR^{3}M^{4}}\left(\|r^{\frac12}\widehat{n}_{0}\|_{L^{2}}^{2}-C^{2}R^{3}M^{4} \right)\\\leq&-\frac{\|r^{\frac12}\widehat{n}_{0}\|_{L^{2}}^{4}}{CAR^{3}M^{4}}\left(\|r^{\frac12}\widehat{n}_{0}\|_{L^{2}}^{2}-G(t)-C^{2}R^{3}M^{4} \right).
			\end{aligned}
		\end{equation*}
		By  contradiction, it can be concluded that
		\begin{equation}\label{h(t)}
			\|r^{\frac12}\widehat{n}_{0}\|_{L^{2}}^{2}-G(t)\leq \big\|r^{\frac12}\widehat{(n_{\rm in})}_{0}\big\|_{L^{2}}^{2}+ 2C^{2}R^{3}M^{4}.
		\end{equation}
		Using elliptic estimates similar to Lemma \ref{lem:varphi k>0}, we have 
		\begin{equation}\label{varphi neq infty}
			\begin{aligned}
				\|k\widehat{\varphi}_{k}\|_{L^{\infty}L^{\infty}}&\leq C(R) \|\widehat{w}_{k}\|_{L^{\infty}L^{2}}.
			\end{aligned}
		\end{equation}
		Therefore, we have 
		\begin{equation}\label{phi_n}
			\begin{aligned}
				&\Big\|\sum_{k\in\mathbb{Z},k\neq0}k\widehat{\varphi}_{k}(s,r)\widehat{n}_{-k}(s,r)\Big\|_{L^2L^2}
				\\\leq& \sum_{k\in\mathbb{Z},k\neq0}\|k\widehat{\varphi}_{k}\|_{L^{\infty}L^{\infty}}\|\widehat{n}_{-k}\|_{L^2L^2}\leq C(R) \sum_{k\in\mathbb{Z},k\neq0}\|\widehat{w}_{k}\|_{L^{\infty}L^{2}}\|\widehat{n}_{-k}\|_{L^2L^2}\\
				\leq& C(R) \sum_{k\in\mathbb{Z},k\neq0} A^{\frac16}\|w_k\|_{X_{a}^{k}}\|n_k\|_{X_{a}^{k}}\leq C(R) A^{\frac16}\|w\|_{Y_{a}}\|n\|_{Y_{a}}
				\leq C(R) A^{\frac16}\mathcal{Q}_{1}^2.
			\end{aligned}
		\end{equation}
		Combining Lemma \ref{lem:ck}  with 
		\begin{equation*}
			\|\widehat{n}_{k}\|_{L^{\infty}L^{\infty}}\leq C\|n\|_{L^{\infty}L^{\infty}}\leq C\mathcal{Q}_{2},
		\end{equation*}
		we obtain that  
		\begin{equation}\label{n_c}
			\begin{aligned}
				&\Big\|\sum_{k\in\mathbb{Z},k\neq0}\widehat{n}_{k}(s,r)\widehat{\partial_r c}_{-k}(s,r)\Big\|_{L^{2}L^2}
				\\\leq& \sum_{k\in\mathbb{Z},k\neq0}
				\|\widehat{n}_{k}\|_{L^{\infty}L^{\infty}}
				\|\widehat{\partial_r c}_{-k}\|_{L^{2}L^2}\leq C \mathcal{Q}_{2}\sum_{k\neq 0, k\in\mathbb{Z}}\|{\rm e}^{aA^{-\frac13}|k|^{\frac23}R^{-2}t}\partial_{r}c_{k}\|_{L^{2}L^{2}}\\
				\leq& C\mathcal{Q}_{2}\sum_{k\neq 0, k\in\mathbb{Z}}\|{\rm e}^{aA^{-\frac13}|k|^{\frac23}R^{-2}t}n_{k}\|_{L^{2}L^{2}}\leq CA^{\frac16}\mathcal{Q}_2\|n\|_{Y_{a}} \leq C A^{\frac16}\mathcal{Q}_{1}\mathcal{Q}_{2}.
			\end{aligned}
		\end{equation}
		By \eqref{phi_n} and \eqref{n_c}, when 
		$$A\geq \max\{\mathcal{Q}_{1}^{6},\mathcal{Q}_{2}^{6}, (C\log R)^3\}=:\mathcal{C}_{1},$$ 
		$G(t)$ can be controlled as
		\begin{equation*}
			\begin{aligned}
				G(t)\leq \frac{C(R)\left(\mathcal{Q}_{1}^4+\mathcal{Q}_{2}^4 \right)}{A^{\frac23}}\leq C(R).
			\end{aligned}
		\end{equation*}
		Combining it with (\ref{h(t)}), one deduces
		\begin{equation*}
			\|r^{\frac12}\widehat{n}_{0}\|_{L^{\infty}L^{2}}\leq C(R)\left(\big\|\widehat{(n_{\rm in})}_{0}\big\|_{L^{2}}+M^{2}+1 \right).
		\end{equation*}
		
		\noindent{\bf Estimate (\ref{result D2}).}
		Note that $\widehat{w}_{0}$ satisfies
		\begin{equation*}
			\partial_{t}\widehat{w}_{0}-\frac{1}{A}\Big(\partial_{r}^{2}+\frac{1}{r}\partial_{r} \Big)\widehat{w}_{0}+\frac{1}{Ar}P_{0}(\partial_{r}\varphi\partial_{\theta}w-\partial_{\theta}\varphi\partial_{r}w)=0,\quad \widehat{w}_{0}|_{r=1,R}=0.
		\end{equation*}
		Multiplying the above equation by $r\widehat{w}_{0}$ and integrating  with $r$ over $[1,R]$, we obtain
		\begin{equation*}
			\begin{aligned}
				\langle\partial_{t}\widehat{w}_{0}-\frac{1}{A}\big(\partial_{r}^{2}+\frac{1}{r}\partial_{r}\big)\widehat{w}_{0}, r\widehat{w}_{0}\rangle=\langle -\frac{1}{Ar}P_{0}(\partial_{r}\varphi\partial_{\theta}w-\partial_{\theta}\varphi\partial_{r}w), r\widehat{w}_{0}  \rangle .
			\end{aligned}
		\end{equation*}
		This implies that
		\begin{equation}\label{eatimate w}
			\begin{aligned}
				\frac12\frac{d}{dt}\|r^{\frac12}\widehat{w}_{0}\|_{L^{2}}^{2}+\frac{1}{A}\|r^{\frac12}\partial_{r}\widehat{w}_{0}\|_{L^{2}}^{2}=-\frac{1}{A}\langle P_{0}[\partial_{r}(\varphi\partial_{\theta}w)-\partial_{\theta}(\varphi\partial_{r}w)], \widehat{w}_{0} \rangle=:J_{1}. 
			\end{aligned}
		\end{equation}
		Due to $P_{0}[\partial_{\theta}(\varphi\partial_{r}w)]=\partial_{\theta}\widehat{(\varphi\partial_{r}w)}_{0}=0$, there holds
		\begin{equation*}
			\begin{aligned}
				J_{1}=\frac{1}{A}\langle \widehat{(\varphi\partial_{\theta}w)}_{0}, \partial_{r}\widehat{w}_{0} \rangle&\leq\frac{1}{A}\|r^{-\frac12}\widehat{(\varphi\partial_{\theta}w)}_{0}\|_{L^{2}}\|r^{\frac12}\partial_{r}\widehat{w}_{0}\|_{L^{2}}\\&\leq\frac{1}{2A}\|r^{-\frac12}\widehat{(\varphi\partial_{\theta}w)}_{0}\|_{L^{2}}^{2}+\frac{1}{2A}\|r^{\frac12}\partial_{r}\widehat{w}_{0}\|_{L^{2}}^{2}.
			\end{aligned}
		\end{equation*}
		This along with (\ref{eatimate w}) gives that
		\begin{equation*}
			\begin{aligned}
				\frac{d}{dt}\|r^{\frac12}\widehat{w}_{0}\|_{L^{2}}^{2}+\frac{1}{A}\|r^{\frac12}\partial_{r}\widehat{w}_{0}\|_{L^{2}}^{2}\leq\frac{1}{A}\|r^{-\frac12}\widehat{(\varphi\partial_{\theta}w)}_{0}\|_{L^{2}}^{2}\leq\frac{1}{A}\|\widehat{(\varphi\partial_{\theta}w)}_{0}\|_{L^{2}}^{2}
				.
			\end{aligned}
		\end{equation*}
		Regarding the integration of $t$ and using (\ref{fg zero}), we obtain
		\begin{equation}\label{estimate w0 1}
			\begin{aligned}
				&\quad \|r^{\frac12}\widehat{w}_{0}\|_{L^{2}}^{2}+\frac{1}{A}\int_{0}^{t}\|r^{\frac12}\partial_{r}\widehat{w}_{0}(s)\|_{L^{2}}^{2}ds\\&\leq\big\|r^{\frac12}\widehat{(w_{\rm in})}_{0}\big\|_{L^{2}}^{2}+\frac{1}{A}\int_{0}^{t}
				\Big\|\sum_{k\in\mathbb{Z},k\neq0}k\widehat{\varphi}_{k}(s)\widehat{w}_{-k}(s)\Big\|_{L^2}^2 ds,
			\end{aligned}
		\end{equation}
		where we used 
		\begin{equation*}
			\begin{aligned}
				\widehat{(\varphi \partial_{\theta}w)}_{0}=\sum_{k\in\mathbb{Z}}\widehat{\varphi}_{k}(t,r)\widehat{\partial_{\theta} w}_{-k}(t,r)
				=-\sum_{k\in\mathbb{Z},k\neq0}ik\widehat{\varphi}_{k}(t,r)\widehat{w}_{-k}(t,r).
			\end{aligned}
		\end{equation*}
		By \eqref{varphi neq infty}, there holds 
		\begin{equation}\label{estimate w0}
			\begin{aligned}
				&\Big\|\sum_{k\in\mathbb{Z},k\neq0}k\widehat{\varphi}_{k}(s,r)\widehat{w}_{-k}(s,r)\Big\|_{L^2L^2}
				\leq \sum_{k\in\mathbb{Z},k\neq0}\|k\widehat{\varphi}_{k}\|_{L^{\infty}L^{\infty}}\|\widehat{w}_{-k}\|_{L^2L^2}\\
				&\leq C(R) \sum_{k\in\mathbb{Z},k\neq0}\|\widehat{w}_{k}\|_{L^{\infty}L^{2}}\|\widehat{w}_{-k}\|_{L^2L^2}\leq C(R) A^{\frac16}\|w\|_{Y_{a}}^2
				\leq C(R) A^{\frac16}\mathcal{Q}_{1}^2.
			\end{aligned}
		\end{equation}
		Substituting (\ref{estimate w0}) into (\ref{estimate w0 1}), 
		when $A\geq \mathcal{C}_1$,
		we arrive at
		\begin{equation*}
			\|r^{\frac12}\widehat{w}_{0}\|_{L^{\infty}L^{2}}^{2}+\frac{1}{A}\|r^{\frac12}\partial_{r}\widehat{w}_{0}\|_{L^{2}L^{2}}^{2}
			\leq C(R)\left(\big\|\widehat{(w_{\rm in})}_{0}\big\|_{L^{2}}^{2}+1\right).
		\end{equation*}

		To sum up, the proof is complete.	
	\end{proof}

	\section{The estimate of $E(t)$ and proof of Proposition \ref{prop:E(t)}}\label{estimate E(t)}
	Write the equations satisfied by $n_{k}$ and $w_{k}$ in (\ref{eq:fourier 1}) as
	\begin{equation}\label{eq:nk wk}
		\left\{
		\begin{array}{lr}
			\partial_{t}n_{k}+\mathcal{L}_{k}n_{k}+\frac{1}{r}\big[ikf_{1}-r^{\frac12}\partial_{r}(r^{\frac12}f_{2}) \big]=0,\\
			\partial_{t}w_{k}+\mathcal{L}_{k}w_{k}+\frac{1}{r}\big[ikg_{1}-r^{\frac12}\partial_{r}(r^{\frac12}g_{2}) \big]=0,\\
			n_{k}|_{r=1,R}=0,\quad w_{k}|_{r=1,R}=0,
		\end{array}
		\right.
	\end{equation}
	where 
	\begin{equation}\label{f1 f2}
		\begin{aligned}
			f_{1}&=\frac{1}{A}\sum_{l\in\mathbb{Z}}\partial_{r}(r^{-\frac12}\varphi_{l})n_{k-l}+\frac{1}{A}\sum_{l\in\mathbb{Z}}i(k-l)r^{-\frac32}n_{l}c_{k-l},\\
			f_{2}&=\frac{1}{A}\sum_{l\in\mathbb{Z}}ilr^{-\frac32}\varphi_{l}n_{k-l}-\frac{1}{A}\sum_{l\in\mathbb{Z}}n_{l}\partial_{r}(r^{-\frac12}c_{k-l}),\\
			g_{1}&=\frac{1}{A}\sum_{l\in\mathbb{Z}}\partial_{r}(r^{-\frac12}\varphi_{l})w_{k-l}+\frac{1}{A}n_{k},\\
			g_{2}&=\frac{1}{A}\sum_{l\in\mathbb{Z}}ilr^{-\frac32}\varphi_{l}w_{k-l}.
		\end{aligned}
	\end{equation}
	The following lemma provides the estimates of the nonlinear terms $f_{1}, f_{2}, g_{1}$ and $g_{2}$.
	\begin{lemma}\label{lem:f1 f2 g1 g2}
		There hold
		\begin{itemize}
			\item[(i)]
			\begin{equation*}
				\begin{aligned}
					\|f_{1}\|_{L^{2}}&\leq \frac{C(R)}{A}\Big(\sum_{l\in\mathbb{Z}\backslash\{0,k\}}|l|^{-\frac12}\|w_{l}\|_{L^{2}}\|n_{k-l}\|_{L^{2}}+|k|^{-\frac12}\|w_{k}\|_{L^{2}}\|\widehat{n}_{0}\|_{L^{2}}+\|\widehat{w}_{0}\|_{L^{2}}\|n_{k}\|_{L^{2}} \Big)\\&\quad+\frac{1}{A}\sum_{l\in\mathbb{Z}\backslash\{0,k\}}|k-l|\|c_{k-l}\|_{L^{\infty}}\|n_{l}\|_{L^{2}}+\frac{C(R)}{A}|k|\|c_{k}\|_{L^{2}}\|\widehat{n}_{0}\|_{L^{\infty}},
				\end{aligned}
			\end{equation*} 
			\item[(ii)] 
			\begin{equation*}
				\begin{aligned}
					\|f_{2}\|_{L^{2}}&\leq\frac{1}{A}\sum_{l\in\mathbb{Z}\backslash\{0,k\}}|l|\|\varphi_{l}\|_{L^{\infty}}\|n_{k-l}\|_{L^{2}}+\frac{C(R)}{A}|k|\|\varphi_{k}\|_{L^{\infty}}\|\widehat{n}_{0}\|_{L^{2}}\\&\quad+\frac{1}{A}\sum_{l\in\mathbb{Z}\backslash\{0,k\}}\|n_{l}\|_{L^{\infty}}\|\partial_{r}(r^{-\frac12}c_{k-l})\|_{L^{2}}+\frac{C(R)}{A}\|\widehat{n}_{0}\|_{L^{\infty}}\|\partial_{r}(r^{-\frac12}c_{k})\|_{L^{2}}\\&\quad+\frac{1}{A}\|n_{k}\|_{L^{2}}\|\partial_{r}(r^{-\frac12}c_{0})\|_{L^{\infty}},
				\end{aligned}
			\end{equation*}
			\item[(iii)]
			\begin{equation*}
				\begin{aligned}
					\|g_{1}\|_{L^{2}}&\leq\frac{C(R)}{A}\Big(\sum_{l\in\mathbb{Z}\backslash\{0,k\}}|l|^{-\frac12}\|w_{l}\|_{L^{2}}\|w_{k-l}\|_{L^{2}}+\|\widehat{w}_{0}\|_{L^{2}}\|w_{k}\|_{L^{2}} \Big)+\frac{1}{A}\|n_{k}\|_{L^{2}},
				\end{aligned}
			\end{equation*}
			\item[(iv)] 
			\begin{equation*}
				\|g_{2}\|_{L^{2}}\leq\frac{1}{A}\sum_{l\in\mathbb{Z}\backslash\{0,k\}}|l|\|\varphi_{l}\|_{L^{\infty}}\|w_{k-l}\|_{L^{2}}+\frac{C(R)}{A}|k|\|\varphi_{k}\|_{L^{\infty}}\|\widehat{w}_{0}\|_{L^{2}}.
			\end{equation*}
		\end{itemize}
	\end{lemma}
	\begin{proof}
		{\bf Estimate $\|f_{1}\|_{L^{2}}$.}	Recall the expression of $f_{1}$ in $(\ref{f1 f2})_{1}$, we get
		\begin{equation}\label{f1 L2}
			\begin{aligned}
				\|f_{1}\|_{L^{2}}\leq\frac{1}{A}\Big\|\sum_{l\in\mathbb{Z}}\partial_{r}(r^{-\frac12}\varphi_{l})n_{k-l}\Big\|_{L^{2}}+\frac{1}{A}\Big\|\sum_{l\in\mathbb{Z}}(k-l)r^{-\frac32}n_{l}c_{k-l}\Big\|_{L^{2}}.
			\end{aligned}
		\end{equation}
		Using Lemma \ref{lem:varphi k>0}, Lemma \ref{lem:varphi k=0} and (\ref{define nk ck wk}), direct calculations show that
		\begin{equation}\label{varphi l}
			\begin{aligned}
				\|\partial_{r}(r^{-\frac12}\varphi_{l})\|_{L^{\infty}}&\leq\frac12\|r^{-\frac32}\varphi_{l}\|_{L^{\infty}}+\|r^{-\frac12}\partial_{r}\varphi_{l}\|_{L^{\infty}}\leq|l|\|r^{-\frac12}\varphi_{l}\|_{L^{\infty}}+\|r^{\frac12}\partial_{r}\varphi_{l}\|_{L^{\infty}}\\&\leq C(R)|l|^{-\frac12}\|rw_{l}\|_{L^{2}}\leq C(R)|l|^{-\frac12}\|w_{l}\|_{L^{2}}
			\end{aligned}
		\end{equation}
		for $l\in\mathbb{Z}\backslash\{0\} $	and
		\begin{equation}\label{varphi 0}
			\begin{aligned}
				\|\partial_{r}(r^{-\frac12}\varphi_{0})\|_{L^{\infty}}&\leq\frac12\|r^{-\frac32}\varphi_{0}\|_{L^{\infty}}+\|r^{-\frac12}\partial_{r}\varphi_{0}\|_{L^{\infty}}\\&\leq\|\varphi_{0}\|_{L^{\infty}}+\|\partial_{r}\varphi_{0}\|_{L^{\infty}}\leq C(R)\left(\|\widehat{\varphi}_{0}\|_{L^{\infty}}+\|\partial_{r}\widehat{\varphi}_{0}\|_{L^{\infty}} \right)\\&\leq C(R)\|r\widehat{w}_{0}\|_{L^{2}}+C(R)\|r^{\frac32}\widehat{w}_{0}\|_{L^{2}}\leq C(R)\|\widehat{w}_{0}\|_{L^{2}}.
			\end{aligned}
		\end{equation}
		Based on the above estimates and (\ref{define nk ck wk}), we have
		{\small
			\begin{equation*}
				\begin{aligned}
					&\quad\frac{1}{A}\Big\|\sum_{l\in\mathbb{Z}}\partial_{r}(r^{-\frac12}\varphi_{l})n_{k-l}\|_{L^{2}}\\&\leq\frac{1}{A}\Big(\sum_{l\in\mathbb{Z}\backslash\{0,k\}}\|\partial_{r}(r^{-\frac12}\varphi_{l})\|_{L^{\infty}}\|n_{k-l}\|_{L^{2}}+\|\partial_{r}(r^{-\frac12}\varphi_{0})\|_{L^{\infty}}\|n_{k}\|_{L^{2}}+\|\partial_{r}(r^{-\frac12}\varphi_{k})\|_{L^{\infty}}\|n_{0}\|_{L^{2}} \Big)\\&\leq\frac{C(R)}{A}\Big(\sum_{l\in\mathbb{Z}\backslash\{0,k\}}|l|^{-\frac12}\|w_{l}\|_{L^{2}}\|n_{k-l}\|_{L^{2}}+|k|^{-\frac12}\|w_{k}\|_{L^{2}}\|\widehat{n}_{0}\|_{L^{2}}+\|\widehat{w}_{0}\|_{L^{2}}\|n_{k}\|_{L^{2}} \Big).
				\end{aligned}
		\end{equation*}}
		Combining this with  (\ref{f1 L2}) and 
		\begin{equation*}
			\begin{aligned}
				&\quad\frac{1}{A}\Big\|\sum_{l\in\mathbb{Z}}(k-l)r^{-\frac32}n_{l}c_{k-l}\Big\|_{L^{2}}\\&\leq\frac{1}{A}\sum_{l\in\mathbb{Z}\backslash\{0,k\}}\| (k-l)n_{l}c_{k-l}  \|_{L^{2}}+\frac{1}{A}\|kn_{0}c_{k} \|_{L^{2}}\\&\leq\frac{1}{A}\sum_{l\in\mathbb{Z}\backslash\{0,k\}}|k-l|\|c_{k-l}\|_{L^{\infty}}\|n_{l}\|_{L^{2}}+\frac{R^{\frac12}}{A}|k|\| c_{k} \|_{L^{2}}\|\widehat{n}_{0}\|_{L^{\infty}},
			\end{aligned}
		\end{equation*}
		we get the inequality of (i).

		\noindent{\bf Estimate $\|f_{2}\|_{L^{2}}$.} We rewrite $(\ref{f1 f2})_{2}$ into 
		\begin{equation*}
			\begin{aligned}
				f_{2}&=\frac{1}{A}\sum_{l\in\mathbb{Z}\backslash\{0,k\}}ilr^{-\frac32}\varphi_{l}n_{k-l}+\frac{ik}{A}r^{-\frac32}\varphi_{k}n_{0}\\&\quad-\frac{1}{A}\sum_{l\in\mathbb{Z}\backslash\{0,k\}}n_{l}\partial_{r}(r^{-\frac12}c_{k-l})-\frac{1}{A}n_{0}\partial_{r}(r^{-\frac12}c_{k})-\frac{1}{A}n_{k}\partial_{r}(r^{-\frac12}c_{0}).
			\end{aligned}
		\end{equation*}
		Then the $L^{2}$ norm estimation indicates
		\begin{align*}
			\|f_{2}\|_{L^{2}}&\leq\frac{1}{A}\sum_{l\in\mathbb{Z}\backslash\{0,k\}}\|lr^{-\frac32}\varphi_{l}n_{k-l}\|_{L^{2}}+\frac{|k|}{A}\|r^{-\frac32}\varphi_{k}n_{0}\|_{L^{2}}\\&\quad+\frac{1}{A}\sum_{l\in\mathbb{Z}\backslash\{0,k\}}\|n_{l}\partial_{r}(r^{-\frac12}c_{k-l})\|_{L^{2}}+\frac{1}{A}\|n_{0}\partial_{r}(r^{-\frac12}c_{k})\|_{L^{2}}+\frac{1}{A}\|n_{k}\partial_{r}(r^{-\frac12}c_{0})\|_{L^{2}}\\&\leq\frac{1}{A}\sum_{l\in\mathbb{Z}\backslash\{0,k\}}|l|\|\varphi_{l}\|_{L^{\infty}}\|n_{k-l}\|_{L^{2}}+\frac{|k|}{A}\|\varphi_{k}\|_{L^{\infty}}\|n_{0}\|_{L^{2}}\\&\quad+\frac{1}{A}\sum_{l\in\mathbb{Z}\backslash\{0,k\}}\|n_{l}\|_{L^{\infty}}\|\partial_{r}(r^{-\frac12}c_{k-l})\|_{L^{2}}+\frac{1}{A}\|n_{0}\|_{L^{\infty}}\|\partial_{r}(r^{-\frac12}c_{k})\|_{L^{2}}\\&\quad+\frac{1}{A}\|n_{k}\|_{L^{2}}\|\partial_{r}(r^{-\frac12}c_{0})\|_{L^{\infty}}.
		\end{align*}
		This along with (\ref{define nk ck wk}) gives the inequality of (ii).
		
		\noindent{\bf Estimate $\|g_{1}\|_{L^{2}}$.} It follows from $(\ref{f1 f2})_{3}$ that
		\begin{equation*}
			\begin{aligned}
				\|g_{1}\|_{L^{2}}&\leq\frac{1}{A}\sum_{l\in\mathbb{Z}\backslash\{0,k\}}\|\partial_{r}(r^{-\frac12}\varphi_{l})\|_{L^{\infty}}\|w_{k-l}\|_{L^{2}}+\frac{1}{A}\|\partial_{r}(r^{-\frac12}\varphi_{0})\|_{L^{\infty}}\|w_{k}\|_{L^{2}}\\&\quad +\frac{1}{A}\|\partial_{r}(r^{-\frac12}\varphi_{k})\|_{L^{\infty}}\|w_{0}\|_{L^{2}}+\frac{1}{A}\|n_{k}\|_{L^{2}}.
			\end{aligned}
		\end{equation*}
		This combination of (\ref{define nk ck wk}) and (\ref{varphi l})-(\ref{varphi 0}) indicates that  (iii) holds true.
		
		\noindent{\bf Estimate $\|g_{2}\|_{L^{2}}.$} According to (\ref{define nk ck wk}) and the definition of $g_{2}$ in $(\ref{f1 f2})_{4}$, we get
		\begin{equation*}
			\begin{aligned}
				\|g_{2}\|_{L^{2}}&\leq\frac{1}{A}\sum_{l\in\mathbb{Z}\backslash\{0,k\}}\|lr^{-\frac32}\varphi_{l}w_{k-l}\|_{L^{2}}+\frac{|k|}{A}\|r^{-\frac32}\varphi_{k}w_{0}\|_{L^{2}}\\&\leq\frac{1}{A}\sum_{l\in\mathbb{Z}\backslash\{0,k\}}|l|\|\varphi_{l}\|_{L^{\infty}}\|w_{k-l}\|_{L^{2}}+\frac{C(R)}{A}|k|\|\varphi_{k}\|_{L^{\infty}}\|\widehat{w}_{0}\|_{L^{2}}.
			\end{aligned}
		\end{equation*}
		
		To sum up, we complete the proof.
		
	\end{proof}
	
	Next, we are committed to estimating the energy functional $E(t)$.
	\begin{proof}[Proof of Proposition \ref{prop:E(t)}]
		{\bf Step I. Estimate $\|n\|_{Y_{a}}$.} 
		When $A\geq (C\log R)^3$,
		by applying Proposition \ref{prop:space time} to $(\ref{eq:nk wk})_{1}$, we obtain
		\begin{equation*}
			\begin{aligned}
				\|n_{k}\|_{X_{a}^{k}}\leq C\left(\|n_{k}(0)\|_{L^{2}}+A^{\frac16}|k|^{-\frac13}\|{\rm e}^{aA^{-\frac13}|k|^{\frac23}R^{-2}t}kf_{1}\|_{L^{2}L^{2}}+A^{\frac12}\|{\rm e}^{aA^{-\frac13}|k|^{\frac23}R^{-2}t}f_{2}\|_{L^{2}L^{2}} \right).
			\end{aligned}
		\end{equation*}
		Regarding the summation over $k\neq 0, k\in\mathbb{Z}$, the above expression indicates that
		\begin{equation}\label{nk Xa}
			\begin{aligned}
				\|n\|_{Y_{a}}
				&\leq C\Big( \sum_{k\neq 0, k\in\mathbb{Z}}\|n_{k}(0)\|_{L^{2}}+\sum_{k\neq 0, k\in\mathbb{Z}}A^{\frac16}|k|^{-\frac13}\|{\rm e}^{aA^{-\frac13}|k|^{\frac23}R^{-2}t}kf_{1}\|_{L^{2}L^{2}}\\&+A^{\frac12}\sum_{k\neq 0, k\in\mathbb{Z}}\|{\rm e}^{aA^{-\frac13}|k|^{\frac23}R^{-2}t}f_{2}\|_{L^{2}L^{2}}\Big)=:C\Big(\sum_{k\neq 0, k\in\mathbb{Z}}\|n_{k}(0)\|_{L^{2}}+T_{1}+T_{2} \Big).
			\end{aligned}
		\end{equation}
		%		It follows from $n_{k}=r^{\frac12}{\rm e}^{ikt}\widehat{n}_{k}$ that
		%			\begin{equation}\label{T1}
			%				T_{1}\leq R^{\frac12}\sum_{k\neq 0, k\in\mathbb{Z}}\|\widehat{n}_{k}(0)\|_{L^{2}}\lesssim R^{\frac12}\|(n_{\rm in})_{\neq}\|_{L^{2}}.
			%			\end{equation}
		Using (i) of Lemma \ref{lem:f1 f2 g1 g2}, $T_{1}$ can be controlled by
		{\small	\begin{equation}\label{T2 0}
				\begin{aligned}
					T_{1}&\leq\frac{C(R)}{A^{\frac56}}\sum_{k\neq 0, k\in\mathbb{Z}}|k|^{\frac23}\sum_{l\in\mathbb{Z}\backslash\{0,k\}}|l|^{-\frac12}\big\|\|{\rm e}^{aA^{-\frac13}|l|^{\frac23}R^{-2}t}w_{l}\|_{L^{2}}
					\|{\rm e}^{aA^{-\frac13}|k-l|^{\frac23}R^{-2}t}n_{k-l}\|_{L^{2}}
					\big\|_{L^{2}} \\&\quad+
					\frac{C(R)}{A^{\frac56}}\sum_{k\neq 0, k\in\mathbb{Z}}|k|^{\frac23}\|\widehat{n}_{0}\|_{L^{\infty}L^{2}}\|{\rm e}^{aA^{-\frac13}|k|^{\frac23}R^{-2}t}w_{k}\|_{L^{2}L^{2}}
					\\&\quad+\frac{C(R)}{A^{\frac56}}\sum_{k\neq 0, k\in\mathbb{Z}}|k|^{\frac23}\|\widehat{w}_{0}\|_{L^{\infty}L^{2}}\|{\rm e}^{aA^{-\frac13}|k|^{\frac23}R^{-2}t}n_{k}\|_{L^{2}L^{2}}\\&\quad+\frac{1}{A^{\frac56}}\sum_{k\neq 0, k\in\mathbb{Z}}|k|^{\frac23}\sum_{l\in\mathbb{Z}\backslash\{0,k\}}|k-l|\big\|\|{\rm e}^{aA^{-\frac13}|l|^{\frac23}R^{-2}t}n_{l}\|_{L^{2}}\|{\rm e}^{aA^{-\frac13}|k-l|^{\frac23}R^{-2}t}c_{k-l}\|_{L^{\infty}} \big\|_{L^{2}}\\&\quad+\frac{C(R)}{A^{\frac56}}\sum_{k\neq 0, k\in\mathbb{Z}}|k|^{\frac53}\|\widehat{n}_{0}\|_{L^{\infty}L^{\infty}}\|{\rm e}^{aA^{-\frac13}|k|^{\frac23}R^{-2}t}c_{k}\|_{L^{2}L^{2}}\\&=:T_{11}+\cdots+T_{15},
				\end{aligned}
		\end{equation}}
		where we use the following inequality
		\begin{equation}\label{k l k-l}
			|k|^{\frac23}\leq |l|^{\frac23}+|k-l|^{\frac23}~~{\rm for~any}~k, l\in\mathbb{Z}~{\rm and~any}~\alpha\in (0,1].
		\end{equation}
		Due to (\ref{assumption}), Lemma \ref{lem:ck} and Lemma \ref{lem:n0 w0}, there holds
		\begin{align*}
			T_{12}&\leq\frac{C(R)}{A^{\frac56}}\|\widehat{n}_{0}\|_{L^{\infty}L^{2}}\sum_{k\neq 0, k\in\mathbb{Z}}\Big\|{\rm e}^{aA^{-\frac13}|k|^{\frac23}R^{-2}t}\frac{kw_{k}}{r}\Big\|_{L^{2}L^{2}}\\&\leq\frac{C(R)}{A^{\frac56}}D_{1}(A^{\frac12}\|w\|_{Y_{a}} )\leq  \frac{C(R)D_{1}\mathcal{Q}_{1}}{A^{\frac13}},\\
			T_{13}&\leq\frac{C(R)}{A^{\frac56}}\|\widehat{w}_{0}\|_{L^{\infty}L^{2}}\sum_{k\neq 0, k\in\mathbb{Z}}\Big\|{\rm e}^{aA^{-\frac13}|k|^{\frac23}R^{-2}t}\frac{kn_{k}}{r}\Big\|_{L^{2}L^{2}}\\&\leq\frac{C(R)}{A^{\frac56}}D_{2}(A^{\frac12}\|n\|_{Y_{a}})\leq \frac{C(R)D_{2}\mathcal{Q}_{1}}{A^{\frac13}}
		\end{align*}
		and
		\begin{equation*}
			\begin{aligned}
				T_{15}&\leq\frac{C(R)}{A^{\frac56}}\|n\|_{L^{\infty}L^{\infty}}\sum_{k\neq 0, k\in\mathbb{Z}}\|{\rm e}^{aA^{-\frac13}|k|^{\frac23}R^{-2}t}k^{2}c_{k}\|_{L^{2}L^{2}}\\&\leq \frac{C(R)\mathcal{Q}_{2}}{A^{\frac56}}\sum_{k\neq 0, k\in\mathbb{Z}}\|{\rm e}^{aA^{-\frac13}|k|^{\frac23}R^{-2}t}n_{k}\|_{L^{2}L^{2}}\leq\frac{C(R)\mathcal{Q}_{2}}{A^{\frac56}}(A^{\frac16}R\|n\|_{Y_{a}})\leq\frac{C(R)\mathcal{Q}_{1}\mathcal{Q}_{2}}{A^{\frac23}}.
			\end{aligned}
		\end{equation*}
		Next we estimate $T_{11}$ and $T_{14}$. 
		The estimate $T_{11}$ is by divided into four cases by discussing the values of $k$ and $l$.

		\noindent{\underline{Case 1: $k>0$, $0\neq l\leq \frac{k}{2}$ or $l\geq\frac32k$. }} Under this circumstance, we have
		\begin{equation}\label{k-l k}
			|k-l|^{-1}\leq C |k|^{-1}.
		\end{equation}
		Let 
		\begin{equation}\label{define A}
			\mathcal{A}=\left\{l\in\mathbb{R}: 0\neq l\leq\frac{k}{2}~{\rm or}~l\geq \frac{3}{2}k \right\}.
		\end{equation}
		Then combining  (\ref{assumption}) with (\ref{k-l k}) and Young inequality for discrete convolution, one obtains
		\begin{align*}
			&\frac{C(R)}{A^{\frac56}}\sum_{k\neq 0, k\in\mathbb{Z}}|k|^{\frac23}\sum_{l\in\mathcal{A}}|l|^{-\frac12}\big\|\|{\rm e}^{aA^{-\frac13}|l|^{\frac23}R^{-2}t}w_{l}\|_{L^{2}}
			\|{\rm e}^{aA^{-\frac13}|k-l|^{\frac23}R^{-2}t}n_{k-l}\|_{L^{2}}
			\big\|_{L^{2}} 
			\\\leq&\frac{C(R)}{A^{\frac12}}\Big(\sum_{k\neq 0, k\in\mathbb{Z}}|k|^{\frac23}|k-l|^{-\frac23}\sum_{l\in\mathcal{A}}|l|^{-\frac12}
			\|{\rm e}^{aA^{-\frac13}|l|^{\frac23}R^{-2}t}w_{l}\|_{L^{\infty}L^{2}}\times\\&
			\big(A^{-\frac16}|k-l|^{\frac13}R^{-1}\|{\rm e}^{aA^{-\frac13}|k-l|^{\frac23}R^{-2}t}n_{k-l}\|_{L^{2}L^{2}} \big)^{\frac12}\big(A^{-\frac12}
			\big\|{\rm e}^{aA^{-\frac13}|k-l|^{\frac23}R^{-2}t}\frac{|k-l|n_{k-l}}{r}\big\|_{L^{2}L^{2}} \big)^{\frac12}
			\Big)\\\leq&\frac{C(R)}{A^{\frac12}}\sum_{k\neq 0, k\in\mathbb{Z}}\sum_{l\in\mathcal{A}}\|w_{l}\|_{X_{a}^{l}}\|n_{k-l}\|_{X_{a}^{k-l}}\\\leq&\frac{C(R)}{A^{\frac12}}\Big(\sum_{k\neq 0, k\in\mathbb{Z}}\|w_{k}\|_{X_{a}^{k}} \Big)\Big(\sum_{k\neq 0, k\in\mathbb{Z}}\|n_{k}\|_{X_{a}^{k}} \Big)\leq\frac{C(R)}{A^{\frac12}}\|w\|_{Y_{a}}\|n\|_{Y_{a}}\leq\frac{C(R)\mathcal{Q}_{1}^{2}}{A^{\frac12}}.
		\end{align*}
		
		\noindent{\underline{Case 2: $k<0$, $ l\leq \frac{3k}{2}$ or $0\neq l\geq\frac{k}{2}$. }} The estimate is similar to Case 1, since (\ref{k-l k}) still holds.
		
		\noindent{\underline{Case 3: $k>0$, $\frac{k}{2}<l<\frac32 k$ but $l\neq k$.}} 
		In this case, $|k|$ is equivalent to $|l|$. Using (\ref{assumption}) and Young inequality for discrete convolution, direct calculations show that
		\begin{equation*}
			\begin{aligned}
				&\frac{C(R)}{A^{\frac56}}\sum_{k\neq 0, k\in\mathbb{Z}}|k|^{\frac23}\sum_{\frac{k}{2}<l<\frac{3}{2}k}|l|^{-\frac12}\big\|\|{\rm e}^{aA^{-\frac13}|l|^{\frac23}R^{-2}t}w_{l}\|_{L^{2}}
				\|{\rm e}^{aA^{-\frac13}|k-l|^{\frac23}R^{-2}t}n_{k-l}\|_{L^{2}}
				\big\|_{L^{2}}
				\\\leq&\frac{C(R)}{A^{\frac56}}\bigg(\sum_{k\neq 0, k\in\mathbb{Z}}|k|^{\frac23}\sum_{\frac{k}{2}<l<\frac{3}{2}k}|l|^{-\frac12}\|{\rm e}^{aA^{-\frac13}|l|^{\frac23}R^{-2}t}w_{l}\|_{L^{2}L^{2}}|{\rm e}^{aA^{-\frac13}|k-l|^{\frac23}R^{-2}t}n_{k-l}\|_{L^{\infty}L^{2}}\bigg)\\\leq&\frac{C(R)}{A^{\frac56}}\sum_{k\neq 0, k\in\mathbb{Z}}|k|^{\frac23}\sum_{\frac{k}{2}<l<\frac{3}{2}k}|l|^{-\frac12}A^{\frac12}|l|^{-1}\|w_{l}\|_{X_{a}^{l}}\|n_{k-l}\|_{X_{a}^{k-l}}\\
				\leq&\frac{C(R)}{A^{\frac13}}\sum_{k\neq 0, k\in\mathbb{Z}}\sum_{\frac{k}{2}<l<\frac{3}{2}k}\|w_{l}\|_{X_{a}^{l}}\|n_{k-l}\|_{X_{a}^{k-l}}\leq\frac{C(R)\mathcal{Q}_{1}^{2}}{A^{\frac13}}.
			\end{aligned}
		\end{equation*}
		
		\noindent{\underline{Case 4: $k<0$, $\frac{3k}{2}<l<\frac{k}{2}$ but $l\neq k$.}} This case is the same as Case 3.
		
		Combining the four cases above, we eventually estimate $T_{11}$ as
		\begin{equation}\label{end T11}
			T_{11}\leq\frac{C(R)\mathcal{Q}_{1}^{2}}{A^{\frac13}}.
		\end{equation}
		Using Gagliardo-Nirenberg inequality
		\begin{equation*}
			\|c_{k-l}\|_{L^{\infty}}\leq C(R) \|c_{k-l}\|_{L^{2}}^{\frac12}\|\partial_{r}c_{k-l}\|_{L^{2}}^{\frac12}
		\end{equation*}
		and Lemma \ref{lem:ck}, $T_{14}$ in (\ref{T2 0}) follows that
		\begin{equation}\label{T14}
			\begin{aligned}
				T_{14}&\leq \frac{C(R)}{A^{\frac56}}
				\bigg(\sum_{k\neq 0, k\in\mathbb{Z}}|k|^{\frac23}\sum_{l\in\mathbb{Z}\backslash\{0,k\}}|k-l|^{-\frac12}\big\|\|{\rm e}^{aA^{-\frac13}|l|^{\frac23}R^{-2}t}n_{l}\|_{L^{2}}
				\\
				&\quad\times\|{\rm e}^{aA^{-\frac13}|k-l|^{\frac23}R^{-2}t}|k-l|^{2}c_{k-l}\|_{L^{2}}^{\frac12}\|{\rm e}^{aA^{-\frac13}|k-l|^{\frac23}R^{-2}t}|k-l|\partial_{r}c_{k-l}\|_{L^{2}}^{\frac12}
				\big\|_{L^{2}}\bigg)
				\\&\leq\frac{C(R)}{A^{\frac56}}\bigg(\sum_{k\neq 0, k\in\mathbb{Z}}|k|^{\frac23}\sum_{l\in\mathbb{Z}\backslash\{0,k\}}|k-l|^{-\frac12}\\&\quad\times\big\|\|{\rm e}^{aA^{-\frac13}|l|^{\frac23}R^{-2}t}n_{l}\|_{L^{2}}\|{\rm e}^{aA^{-\frac13}|k-l|^{\frac23}R^{-2}t}n_{k-l}\|_{L^{2}}  \big\|_{L^{2}}\bigg).
			\end{aligned}
		\end{equation}
		Similar to the estimation of $T_{11}$, we also discuss the four cases of the values of $k$ and $l$ to estimate $T_{14}$.
		
		\noindent{\underline{Case 1: $k>0$, $0\neq l\leq \frac{k}{2}$ or $l\geq\frac32k$. }}  Using (\ref{assumption}), (\ref{k-l k}), (\ref{define A}) and Young inequality for discrete convolution, we get
		\begin{equation*}
			\begin{aligned}
				&\frac{C(R)}{A^{\frac56}}\sum_{k\neq 0, k\in\mathbb{Z}}|k|^{\frac23}\sum_{l\in\mathcal{A}}|k-l|^{-\frac12}\big\|\|{\rm e}^{aA^{-\frac13}|l|^{\frac23}R^{-2}t}n_{l}\|_{L^{2}}\|{\rm e}^{aA^{-\frac13}|k-l|^{\frac23}R^{-2}t}n_{k-l}\|_{L^{2}}  \big\|_{L^{2}}
				\\\leq&\frac{C(R)}{A^{\frac56}}\sum_{k\neq 0, k\in\mathbb{Z}}|k|^{\frac23}\sum_{l\in\mathcal{A}}|k-l|^{-\frac32}\|{\rm e}^{aA^{-\frac13}|l|^{\frac23}R^{-2}t}n_{l}\|_{L^{\infty}L^{2}}\Big\|{\rm e}^{aA^{-\frac13}|k-l|^{\frac23}R^{-2}t}\frac{|k-l|n_{k-l}}{r} \Big\|_{L^{2}L^{2}}\\\leq&\frac{C(R)}{A^{\frac13}}\sum_{k\neq 0, k\in\mathbb{Z}}\sum_{l\in\mathcal{A}}\|n_{l}\|_{X_{a}^{l}}\|n_{k-l}\|_{X_{a}^{k-l}}\leq\frac{C(R)\|n\|^2_{Y_{a}}}{A^{\frac13}}
				\leq\frac{C(R)\mathcal{Q}_{1}^{2}}{A^{\frac13}}.
			\end{aligned}
		\end{equation*} 
		
		\noindent{\underline{Case 2: $k<0$, $ l\leq \frac{3k}{2}$ or $0\neq l\geq\frac{k}{2}$. }} This case is the same as Case 1.		
		
		\noindent{\underline{Case 3: $k>0$, $\frac{k}{2}<l<\frac32 k$ but $l\neq k$.}} By applying (\ref{assumption}) and Young inequality for discrete convolution, there holds
		\begin{equation*}
			\begin{aligned}
				&\frac{C(R)}{A^{\frac56}}\sum_{k\neq 0, k\in\mathbb{Z}}|k|^{\frac23}\sum_{\frac{k}{2}<l<\frac{3}{2}k}|k-l|^{-\frac12}\big\|\|{\rm e}^{aA^{-\frac13}|l|^{\frac23}R^{-2}t}n_{l}\|_{L^{2}}\|{\rm e}^{aA^{-\frac13}|k-l|^{\frac23}R^{-2}t}n_{k-l}\|_{L^{2}}  \big\|_{L^{2}}
				\\\leq&\frac{C(R)}{A^{\frac56}}\sum_{k\neq 0, k\in\mathbb{Z}}|k|^{\frac23}\sum_{\frac{k}{2}<l<\frac{3}{2}k}|l|^{-1}\big\|{\rm e}^{aA^{-\frac13}|l|^{\frac23}R^{-2}t}\frac{|l|n_{l}}{r} \big\|_{L^{2}L^{2}}\|{\rm e}^{aA^{-\frac13}|k-l|^{\frac23}R^{-2}t}n_{k-l}\|_{L^{\infty}L^{2}}
				\\\leq&\frac{C(R)}{A^{\frac13}}\sum_{k\neq 0, k\in\mathbb{Z}}\sum_{\frac{k}{2}<l<\frac{3}{2}k}\|n_{l}\|_{X_{a}^{l}}\|n_{k-l}\|_{X_{a}^{k-l}}
				\leq\frac{C(R)\|n\|_{Y_a}^2}{A^{\frac13}}
				\leq\frac{C(R)\mathcal{Q}_{1}^{2}}{A^{\frac13}}.
			\end{aligned}
		\end{equation*}
		
		\noindent{\underline{Case 4: $k<0$, $\frac{3k}{2}<l<\frac{k}{2}$ but $l\neq k$.}} This case is the same as Case 3.
		
		Based on the four cases above, $T_{14}$ in (\ref{T14}) can be estimated as
		\begin{equation*}
			T_{14}\leq\frac{C(R)\mathcal{Q}_{1}^{2}}{A^{\frac13}}.
		\end{equation*}
		Collecting the estimates of $T_{11}-T_{14}$, (\ref{T2 0}) yields that
		\begin{equation}\label{T2}
			\begin{aligned}
				T_{1}\leq \frac{C(R)\mathcal{Q}_{1}(\mathcal{Q}_{1}+\mathcal{Q}_{2}+D_{1}+D_{2})}{A^{\frac13}}.
			\end{aligned}
		\end{equation}
		
		According to (ii) of Lemma \ref{lem:f1 f2 g1 g2}, (\ref{nk Xa}) and (\ref{k l k-l}), we arrive at
		\begin{align}
			T_{2}&\leq\frac{1}{A^{\frac12}}\sum_{k\neq 0, k\in\mathbb{Z}}\sum_{l\in\mathbb{Z}\backslash\{0,k\}}|l|\|{\rm e}^{aA^{-\frac13}|l|^{\frac23}R^{-2}t}\varphi_{l}\|_{L^{2}L^{\infty}}\|{\rm e}^{aA^{-\frac13}|k-l|^{\frac23}R^{-2}t}n_{k-l}\|_{L^{\infty}L^{2}}
			\nonumber
			\\&\quad+\frac{C(R)}{A^{\frac12}}\|\widehat{n}_{0}\|_{L^{\infty}L^{2}}\sum_{k\neq 0, k\in\mathbb{Z}}|k|\|{\rm e}^{aA^{-\frac13}|k|^{\frac23}R^{-2}t}\varphi_{k}\|_{L^{2}L^{\infty}}\nonumber\\&\quad+\frac{1}{A^{\frac12}}\sum_{k\neq 0, k\in\mathbb{Z}}\sum_{l\in\mathbb{Z}\backslash\{0,k\}}\|{\rm e}^{aA^{-\frac13}|l|^{\frac23}R^{-2}t}n_{l}\|_{L^{2}L^{\infty}}\|{\rm e}^{aA^{-\frac13}|k-l|^{\frac23}R^{-2}t}\partial_{r}(r^{-\frac12}c_{k-l})\|_{L^{\infty}L^{2}}\nonumber\\&\quad+\frac{C(R)}{A^{\frac12}}\|\widehat{n}_{0}\|_{L^{\infty}L^{\infty}}\sum_{k\neq 0, k\in\mathbb{Z}}\|{\rm e}^{aA^{-\frac13}|k|^{\frac23}R^{-2}t}\partial_{r}(r^{-\frac12}c_{k})\|_{L^{2}L^{2}}\nonumber\\&\quad+\frac{1}{A^{\frac12}}\|\partial_{r}(r^{-\frac12}c_{0})\|_{L^{\infty}L^{\infty}}\sum_{k\neq 0, k\in\mathbb{Z}}\|{\rm e}^{aA^{-\frac13}|k|^{\frac23}R^{-2}t}n_{k}\|_{L^{2}L^{2}}=:T_{21}+\cdots+T_{25}.\label{T3 0}
		\end{align}
		For $T_{21}$, using (\ref{assumption}), Lemma \ref{lem:varphi k>0} and Young inequality for discrete convolution, we get
		\begin{equation*}
			\begin{aligned}
				T_{21}&\leq \frac{C(R)}{A^{\frac12}}\sum_{k\neq 0, k\in\mathbb{Z}}\sum_{l\in\mathbb{Z}\backslash\{0,k\}}\|{\rm e}^{aA^{-\frac13}|l|^{\frac23}R^{-2}t}w_{l}\|_{L^{2}L^{2}}\|{\rm e}^{aA^{-\frac13}|k-l|^{\frac23}R^{-2}t}n_{k-l}\|_{L^{\infty}L^{2}}\\&\leq\frac{C(R)}{A^{\frac13}}\sum_{k\neq 0, k\in\mathbb{Z}}\sum_{l\in\mathbb{Z}\backslash\{0,k\}}\|w_{l}\|_{X_{a}^{l}}\|n_{k-l}\|_{X_{a}^{k-l}}\\&\leq\frac{C(R)}{A^{\frac13}}\Big(\sum_{k\neq 0, k\in\mathbb{Z}}\|w_{k}\|_{X_{a}^{k}} \Big)\Big(\sum_{k\neq 0, k\in\mathbb{Z}}\|n_{k}\|_{X_{a}^{k}} \Big)\leq\frac{C(R)\mathcal{Q}_{1}^{2}}{A^{\frac13}}.
			\end{aligned}
		\end{equation*}
		Similarly, by (\ref{assumption}), Lemma \ref{lem:varphi k>0} and Lemma \ref{lem:n0 w0}, there holds
		\begin{equation*}
			\begin{aligned}
				T_{22}&\leq\frac{C(R)D_{1}}{A^{\frac12}}\sum_{k\neq 0, k\in\mathbb{Z}}\|{\rm e}^{aA^{-\frac13}|k|^{\frac23}R^{-2}t}w_{k}\|_{L^{2}L^{2}}\\&\leq\frac{C(R)D_{1}}{A^{\frac13}}\sum_{k\neq 0, k\in\mathbb{Z}}\|w_{k}\|_{X_{a}^{k}}\leq\frac{C(R)D_{1}\mathcal{Q}_{1}}{A^{\frac13}}.
			\end{aligned}
		\end{equation*}
		Using Lemma \ref{lem:ck}, for $|k|\geq 1$, the direct calculation indicates that
		\begin{equation*}
			\begin{aligned}
				\|\partial_{r}(r^{-\frac12}c_{k})\|_{L^{2}}\leq C\left(\|c_{k}\|_{L^{2}}+\|\partial_{r}c_{k}\|_{L^{2}} \right)  \leq C\|n_{k}\|_{L^{2}}.
			\end{aligned}
		\end{equation*}
		Combining this with (\ref{assumption}),  Gagliardo-Nirenberg inequality and Young inequality for discrete convolution, one deduces
		\begin{equation*}
			\begin{aligned}
				T_{23}&\leq\frac{C}{A^{\frac12}}\bigg(\sum_{k\neq 0, k\in\mathbb{Z}}\sum_{l\in\mathbb{Z}\backslash\{0,k\}}\|{\rm e}^{aA^{-\frac13}|l|^{\frac23}R^{-2}t}n_{l}\|_{L^{2}L^{2}}^{\frac12}\|{\rm e}^{aA^{-\frac13}|l|^{\frac23}R^{-2}t}\partial_{r}n_{l}\|_{L^{2}L^{2}}^{\frac12}\\&\quad \times\|{\rm e}^{aA^{-\frac13}|k-l|^{\frac23}R^{-2}t}n_{k-l}\|_{L^{\infty}L^{2}}	\bigg)\\&\leq\frac{C(R)}{A^{\frac16}}\sum_{k\neq 0, k\in\mathbb{Z}}\sum_{l\in\mathbb{Z}\backslash\{0,k\}}\|n_{l}\|_{X_{a}^{l}}\|n_{k-l}\|_{X_{a}^{k-l}}\leq\frac{C(R)\mathcal{Q}_{1}^{2}}{A^{\frac16}}
			\end{aligned}
		\end{equation*}
		and
		\begin{equation*}
			\begin{aligned}
				T_{24}\leq\frac{C(R)}{A^{\frac12}}\|n\|_{L^{\infty}L^{\infty}}\sum_{k\neq 0, k\in\mathbb{Z}}\|{\rm e}^{aA^{-\frac13}|k|^{\frac23}R^{-2}t}n_{k}\|_{L^{2}L^{2}}\leq\frac{C(R)\mathcal{Q}_{1}\mathcal{Q}_{2}}{A^{\frac13}}.
			\end{aligned}
		\end{equation*}
		For $T_{25}$,  noting that
		\begin{equation*}
			\begin{aligned}
				\|\partial_{r}(r^{-\frac12}c_{0})\|_{L^{\infty}L^{\infty}}\leq&\|(1,\partial_{r})c_{0}\|_{L^{\infty}L^{\infty}}\leq C(R)\|(1,\partial_{r})\widehat{c}_{0}\|_{L^{\infty}L^{\infty}}\leq C(R)\|\widehat{n}_{0}\|_{L^{\infty}L^{2}}\\\leq& C(R)\|\widehat{n}_{0}\|_{L^{\infty}L^{1}}^{\frac12}\|\widehat{n}_{0}\|_{L^{\infty}L^{\infty}}^{\frac12}\leq C(R)(M+\mathcal{Q}_{2}),
			\end{aligned}
		\end{equation*}		
		by Lemma \ref{lem:c0}, (\ref{define nk ck wk}) and (\ref{assumption}), we arrive at
		\begin{equation*}
			\begin{aligned}
				T_{25}\leq\frac{C(R)(m+\mathcal{Q}_{2})}{A^{\frac12}}\sum_{k\neq 0, k\in\mathbb{Z}}A^{\frac16}R\|n_{k}\|_{X_{a}^{k}}\leq\frac{C(R)\mathcal{Q}_{1}(M+\mathcal{Q}_{2})}{A^{\frac13}}.
			\end{aligned}
		\end{equation*}
		Collecting the estimates of $T_{21}-T_{25}$, (\ref{T3 0}) shows that
		\begin{equation}\label{T3}
			\begin{aligned}
				T_{2}\leq\frac{C(R)\mathcal{Q}_{1}(\mathcal{Q}_{1}+\mathcal{Q}_{2}+M+D_{1})}{A^{\frac16}}.
			\end{aligned}
		\end{equation}
		
		Combining (\ref{T2}) and (\ref{T3}), we get from (\ref{nk Xa}) that
		\begin{equation}\label{nk end}
			\begin{aligned}
				\|n\|_{Y_{a}}\leq&C\sum_{k\neq 0, k\in\mathbb{Z}}\|n_{k}(0)\|_{L^{2}}
				+\frac{C(R) \mathcal{Q}_{1}(\mathcal{Q}_{1}+\mathcal{Q}_{2}+M+D_{1}+D_{2})}{A^{\frac16}}.
			\end{aligned}
		\end{equation}
		
		\noindent{\bf Step II. Estimate $\|w\|_{Y_{a}}$.} 
		When $A\geq (C\log R)^3$,
		by applying Proposition \ref{prop:space time} to $(\ref{eq:nk wk})_{2}$, we get
		\begin{equation*}
			\|w_{k}\|_{X_{a}^{k}}\leq C\left(\|w_{k}(0)\|_{L^{2}}+A^{\frac16}|k|^{-\frac13}\|{\rm e}^{aA^{-\frac13}|k|^{\frac23}R^{-2}t}kg_{1}\|_{L^{2}L^{2}}+A^{\frac12}\|{\rm e}^{aA^{-\frac13}|k|^{\frac23}R^{-2}t}g_{2}\|_{L^{2}L^{2}} \right).
		\end{equation*}
		After the summation of $k\neq 0, k\in\mathbb{Z}$, the above inequality follows that
		\begin{equation}\label{wk Xa}
			\begin{aligned}
				\|w\|_{Y_{a}}&\leq C\Big( \sum_{k\neq 0, k\in\mathbb{Z}}\|w_{k}(0)\|_{L^{2}}+\sum_{k\neq 0, k\in\mathbb{Z}}A^{\frac16}|k|^{-\frac13}\|{\rm e}^{aA^{-\frac13}|k|^{\frac23}R^{-2}t}kg_{1}\|_{L^{2}L^{2}}\\&+A^{\frac12}\sum_{k\neq 0, k\in\mathbb{Z}}\|{\rm e}^{aA^{-\frac13}|k|^{\frac23}R^{-2}t}g_{2}\|_{L^{2}L^{2}}\Big)=:C\Big(\sum_{k\neq 0, k\in\mathbb{Z}}\|w_{k}(0)\|_{L^{2}}+S_{1}+S_{2} \Big).
			\end{aligned}	
		\end{equation}		
		%As $w_{k}=r^{\frac12}{\rm e}^{ikt}\widehat{w}_{k}$, there holds
		%\begin{equation}\label{S1}
		%	S_{1}\leq R^{\frac12}\sum_{k\neq 0, k\in\mathbb{Z}}\|\widehat{w}_{k}(0)\|_{L^{2}}\lesssim R^{\frac12}\|(w_{\rm in})_{\neq}\|_{L^{2}}.
		%\end{equation}			
		Using (iii) of Lemma \ref{lem:f1 f2 g1 g2}, (\ref{k l k-l}) and (\ref{wk Xa}), we obtain
		\begin{equation}\label{S2 0}
			\begin{aligned}
				S_{1}\leq&\frac{C(R)}{A^{\frac56}}\sum_{k\neq 0, k\in\mathbb{Z}}|k|^{\frac23}\sum_{l\in\mathbb{Z}\backslash\{0,k\}}|l|^{-\frac12}\big\|\|{\rm e}^{aA^{-\frac13}|l|^{\frac23}R^{-2}t}w_{l}\|_{L^{2}}\|{\rm e}^{aA^{-\frac13}|k-l|^{\frac23}R^{-2}t}w_{k-l}\|_{L^{2}}\big\|_{L^{2}}\\&+\frac{C(R)}{A^{\frac56}}\|\widehat{w}_{0}\|_{L^{\infty}L^{2}}\sum_{k\neq 0, k\in\mathbb{Z}}|k|^{\frac23}\|{\rm e}^{aA^{-\frac13}|k|^{\frac23}R^{-2}t}w_{k}\|_{L^{2}L^{2}}\\&+\frac{1}{A^{\frac56}}\sum_{k\neq 0, k\in\mathbb{Z}}|k|^{\frac23}\|{\rm e}^{aA^{-\frac13}|k|^{\frac23}R^{-2}t}n_{k}\|_{L^{2}L^{2}}=:S_{11}+S_{12}+S_{13}.
			\end{aligned}
		\end{equation}
		The estimate of	$S_{11}$ is similar to $T_{11}$. As the estimate of  (\ref{end T11}), by categorically discussing the values of $k$ and $l$, we ultimately obtain
		\begin{equation*}
			S_{11}\leq\frac{C(R)\mathcal{Q}_{1}^{2}}{A^{\frac13}}.
		\end{equation*}
		Due to (\ref{assumption}) and Lemma \ref{lem:n0 w0}, there holds
		\begin{equation*}
			\begin{aligned}
				S_{12}\leq&\frac{C(R)}{A^{\frac56}}\|\widehat{w}_{0}\|_{L^{\infty}L^{2}}\sum_{k\neq 0, k\in\mathbb{Z}}\Big\|{\rm e}^{aA^{-\frac13}|k|^{\frac23}R^{-2}t}\frac{|k|w_{k}}{r}\Big\|_{L^{2}L^{2}}\\\leq&\frac{C(R)}{A^{\frac56}}D_{2}(A^{\frac12}\|w\|_{Y_{a}})\leq\frac{C(R)D_{2}\mathcal{Q}_{1}}{A^{\frac13}}.
			\end{aligned}
		\end{equation*}
		Using (\ref{assumption}), we arrive at
		\begin{equation*}
			\begin{aligned}
				S_{13}\leq\frac{R}{A^{\frac56}}\sum_{k\neq 0, k\in\mathbb{Z}}\Big\|{\rm e}^{aA^{-\frac13}|k|^{\frac23}R^{-2}t}\frac{kn_{k}}{r} \Big\|_{L^{2}L^{2}}\leq\frac{R}{A^{\frac13}}\sum_{k\neq 0, k\in\mathbb{Z}}\|n_{k}\|_{X_{a}^{k}}\leq\frac{C(R)\mathcal{Q}_{1}}{A^{\frac13}}.
			\end{aligned}
		\end{equation*}
		Substituting the estimates of $S_{11}-S_{13}$ into (\ref{S2 0}), one deduces
		\begin{equation}\label{S2}
			\begin{aligned}
				S_{1}\leq\frac{\mathcal{C}(R)\mathcal{Q}_{1}(\mathcal{Q}_{1}+D_{2}+1)}{A^{\frac13}}.
			\end{aligned}
		\end{equation}
		According to (iv) of Lemma \ref{lem:f1 f2 g1 g2}, (\ref{k l k-l}) and (\ref{wk Xa}), $S_{2}$ can be controlled by
		\begin{equation*}
			\begin{aligned}
				S_{2}&\leq\frac{1}{A^{\frac12}}\sum_{k\neq 0, k\in\mathbb{Z}}\sum_{l\in\mathbb{Z}\backslash\{0,k\}}|l|\|{\rm e}^{aA^{-\frac13}|l|^{\frac23}R^{-2}t}\varphi_{l}\|_{L^{2}L^{\infty}}\|{\rm e}^{aA^{-\frac13}|k-l|^{\frac23}R^{-2}t}w_{k-l}\|_{L^{\infty}L^{2}}\\&\quad+\frac{C(R)}{A^{\frac12}}\|\widehat{w}_{0}\|_{L^{\infty}L^{2}}\sum_{k\neq 0, k\in\mathbb{Z}}|k|\|{\rm e}^{aA^{-\frac13}|k|^{\frac23}R^{-2}t}\varphi_{k}\|_{L^{2}L^{\infty}}.
			\end{aligned}
		\end{equation*}			
		Then using (\ref{assumption}), Lemma \ref{lem:varphi k>0}, Lemma \ref{lem:n0 w0} and Young inequality for discrete convolution, we get
		\begin{align}
			S_{2}&\leq\frac{C(R)}{A^{\frac12}}\sum_{k\neq 0, k\in\mathbb{Z}}\sum_{l\in\mathbb{Z}\backslash\{0,k\}}\|{\rm e}^{aA^{-\frac13}|l|^{\frac23}R^{-2}t}w_{l}\|_{L^{2}L^{2}}\|{\rm e}^{aA^{-\frac13}|k-l|^{\frac23}R^{-2}t}w_{k-l}\|_{L^{\infty}L^{2}}
			\nonumber
			\\&\quad+\frac{C(R)D_{2}}{A^{\frac12}}\sum_{k\neq 0, k\in\mathbb{Z}}\|{\rm e}^{aA^{-\frac13}|k|^{\frac23}R^{-2}t}w_{k}\|_{L^{2}L^{2}}
			\nonumber\\
			&\leq\frac{C(R)}{A^{\frac13}}\Big(\sum_{k\neq 0, k\in\mathbb{Z}}\sum_{l\in\mathbb{Z}\backslash\{0,k\}}\|w_{l}\|_{X_{a}^{l}}\|w_{k-l}\|_{X_{a}^{k-l}}+D_{2}\sum_{k\neq 0, k\in\mathbb{Z}}\|w_{k}\|_{X_{a}^{k}} \Big)
			\nonumber\\
			&\leq\frac{C(R)}{A^{\frac13}}\left(\|w\|_{Y_{a}}^{2}+D_{2}\|w\|_{Y_{a}}\right)\leq\frac{C(R)\mathcal{Q}_{1}(\mathcal{Q}_{1}+D_{2})}{A^{\frac13}}.\label{S3}
		\end{align}
		Collecting the estimates of (\ref{S2}) and (\ref{S3}), (\ref{wk Xa}) yields that
		\begin{equation}\label{wk end}
			\begin{aligned}
				\|w\|_{Y_{a}}\leq C\sum_{k\neq 0, k\in\mathbb{Z}}\|w_{k}(0)\|_{L^{2}}+\frac{C(R)\mathcal{Q}_{1}(\mathcal{Q}_{1}+D_{2}+1)}{A^{\frac13}}.
			\end{aligned}
		\end{equation}
		
		Therefore, combining (\ref{nk end}) and (\ref{wk end}), we conclude that
		\begin{equation}\label{E(t) estimate}
			\begin{aligned}
				E(t)\leq CE_{\rm in}+\frac{C(R)\mathcal{Q}_{1}(\mathcal{Q}_{1}+\mathcal{Q}_{2}+M+D_{1}+D_{2}+1)}{A^{\frac16}},
			\end{aligned}
		\end{equation}
		where 
		\begin{equation*}
			E_{\rm in}=\sum_{k\neq 0, k\in\mathbb{Z}}\|n_{k}(0)\|_{L^{2}}+\sum_{k\neq 0, k\in\mathbb{Z}}\|w_{k}(0)\|_{L^{2}}.
		\end{equation*}
		Recalling that $n_{k}=r^{\frac12}{\rm e}^{ikt}\widehat{n}_{k}, w_{k}=r^{\frac12}{\rm e}^{ikt}\widehat{w}_{k}$, and using H$\ddot{\rm o}$lder's inequality, we get
		\begin{equation}\label{Ein}
			\begin{aligned}
				E_{\rm in}&\leq R^{\frac12}\Big[\Big(\sum_{k\neq 0, k\in\mathbb{Z}}\|k\widehat{n}_{k}(0)\|_{L^{2}}^{2} \Big)^{\frac12}+\Big(\sum_{k\neq 0, k\in\mathbb{Z}}\|k\widehat{w}_{k}(0)\|_{L^{2}}^{2} \Big)^{\frac12} \Big]\Big(\sum_{k\neq 0, k\in\mathbb{Z}}\frac{1}{k^{2}} \Big)^{\frac12}\\&\leq C(R)\left(\|\partial_{\theta}n_{\rm in}\|_{L^{2}}+\|\partial_{\theta}w_{\rm in}\|_{L^{2}} \right).
			\end{aligned}
		\end{equation}
		Let us denote 
		\begin{equation*}
			\mathcal{C}_{2}:=\max\{(C\log R)^3, \mathcal{Q}_{1}^{6}(\mathcal{Q}_{1}+\mathcal{Q}_{2}+M+D_{1}+D_{2}+1)^{6}\}.
		\end{equation*}
		Then if $A\geq \mathcal{C}_{2}$, (\ref{E(t) estimate}) and (\ref{Ein}) imply that
		\begin{equation*}
			E(t)\leq C(R)\left(\|\partial_{\theta}n_{\rm in}\|_{L^{2}}+\|\partial_{\theta}w_{\rm in}\|_{L^{2}}+1 \right)=:\mathcal{Q}_{1}.
		\end{equation*}
		
		The proof is complete.
	\end{proof}

	\section{The $L^{\infty}$ estimate of the density and proof of Proposition \ref{prop:n infty}}\label{estimate infty}
	\begin{proof}[Proof of Proposition \ref{prop:n infty}]	Multiplying $(\ref{ini1})_{1}$ by $2pn^{2p-1}$ with $p=2^{j} ~(j\geq 1)$, and integrating by parts the resulting equation over $[1,R]\times\mathbb{S}^{1}$, one obtains
		\begin{align}
			&\frac{d}{dt}\|n^{p}\|_{L^{2}}^{2}+\frac{2(2p-1)}{Ap}\Big\|\Big(\partial_{r}, \frac{1}{r}\partial_{\theta}\Big)n^{p}\Big\|_{L^{2}}^{2}\nonumber\\=&\frac{1}{A}\Big\|\frac{n^{p}}{r} \Big\|_{L^{2}}^{2}+\frac{1}{A}\int_{0}^{2\pi}\int_{1}^{R}\frac{1}{r^{2}}\partial_{\theta}\varphi n^{2p}dr d\theta+\frac{4p}{A}\int_{0}^{2\pi}\int_{1}^{R}\frac{1}{r}n^{p}c\partial_{r}n^{p}drd\theta\nonumber\\&-\frac{2p}{A}\int_{0}^{2\pi}\int_{1}^{R}\frac{1}{r^{2}}n^{2p}cdrd\theta+\frac{2(2p-1)}{A}\int_{0}^{2\pi}\int_{1}^{R}n^{p}\Big(\partial_{r}, \frac{1}{r}\partial_{\theta} \Big)c\cdot\Big(\partial_{r}, \frac{1}{r}\partial_{\theta} \Big)n^{p}drd\theta\nonumber\\\leq&\frac{1}{A}\left(1+\|\partial_{\theta}\varphi\|_{L^{\infty}L^{\infty}}+2p\|c\|_{L^{\infty}L^{\infty}} \right)\|n^{p}\|_{L^{2}}^{2}+\frac{4p}{A}\|n^{p}c\|_{L^{2}}\|\partial_{r}n^{p}\|_{L^{2}}\nonumber\\&+\frac{2(2p-1)}{A}\Big\|n^{p}\Big(\partial_{r},\frac{1}{r}\partial_{\theta}\Big)c\Big\|_{L^{2}}\Big\|\Big(\partial_{r},\frac{1}{r}\partial_{\theta}\Big)n^{p}  \Big\|\nonumber\\\leq&\frac{1}{A}\left(1+\|\partial_{\theta}\varphi\|_{L^{\infty}L^{\infty}}+2p\|c\|_{L^{\infty}L^{\infty}} \right)\|n^{p}\|_{L^{2}}^{2}+\frac{Cp^{2}}{A}\Big\|n^{p}\Big(1,\partial_{r},\frac{1}{r}\partial_{\theta} \Big)c\Big\|_{L^{2}}^{2}\nonumber\\&+\frac{2p-1}{Ap}\Big\|\Big(\partial_{r},\frac{1}{r}\partial_{\theta} \Big)n^{p} \Big\|_{L^{2}}^{2}.\label{np 1}
		\end{align}
		Due to (\ref{assumption}), Lemma \ref{lem:varphi k>0} and $\varphi_{k}=r^{\frac12}{\rm e}^{ikt}\widehat{\varphi}_{k}$, there holds
		\begin{equation*}
			\begin{aligned}
				\|\partial_{\theta}\varphi\|_{L^{\infty}L^{\infty}}\leq&\sum_{k\neq 0, k\in\mathbb{Z}}|k|\|{\rm e}^{aA^{-\frac13}|k|^{\frac23}R^{-2}t}\widehat{\varphi}_{k}\|_{L^{\infty}L^{\infty}}\\\leq& R^{\frac12}\sum_{k\neq 0, k\in\mathbb{Z}}|k|\|{\rm e}^{aA^{-\frac13}|k|^{\frac23}R^{-2}t}r^{-\frac12}\varphi_{k}\|_{L^{\infty}L^{\infty}}\\\leq &C(R)\sum_{k\neq 0, k\in\mathbb{Z}}|k|^{-\frac12}\|{\rm e}^{aA^{-\frac13}|k|^{\frac23}R^{-2}t}rw_{k}\|_{L^{\infty}L^{2}}\\\leq&C(R)\sum_{k\neq 0, k\in\mathbb{Z}}\|{\rm e}^{aA^{-\frac13}|k|^{\frac23}R^{-2}t}w_{k}\|_{L^{\infty}L^{2}}\leq C(R)\mathcal{Q}_{1}.
			\end{aligned}
		\end{equation*}
		Using (\ref{assumption}), Lemma \ref{lem:ck}, Lemma \ref{lem:c0} and Lemma \ref{lem:n0 w0}, we get
		\begin{equation*}
			\begin{aligned}
				\|c\|_{L^{\infty}L^{\infty}}\leq&\|\widehat{c}_{0}\|_{L^{\infty}L^{\infty}}+\sum_{k\neq 0, k\in\mathbb{Z}}\|{\rm e}^{aA^{-\frac13}|k|^{\frac23}R^{-2}t}c_{k}\|_{L^{\infty}L^{\infty}}\\\leq&C(R)\Big(\|r^{\frac12}\widehat{n}_{0}\|_{L^{\infty}L^{2}}+\sum_{k\neq 0, k\in\mathbb{Z}}\|{\rm e}^{aA^{-\frac13}|k|^{\frac23}R^{-2}t}n_{k}\|_{L^{\infty}L^{2}} \Big)\leq C(R)\left(D_{1}+\mathcal{Q}_{1} \right).
			\end{aligned}
		\end{equation*}
		Moreover, it follows from H$\ddot{\rm o}$lder's and Nash inequalities that
		\begin{equation*}
			\begin{aligned}
				\Big\|n^{p}\Big(1,\partial_{r}, \frac{1}{r}\partial_{\theta} \Big)c  \Big\|_{L^{2}}^{2}\leq&\|n^{p}\|_{L^{4}}^{2}\Big\|\Big(1,\partial_{r},\frac{1}{r}\partial_{\theta} \Big)c \Big\|_{L^{4}}^{2}\\\leq& C\|n^{p}\|_{L^{2}}\|\nabla n^{p}\|_{L^{2}}\Big\|\Big(1,\partial_{r},\frac{1}{r}\partial_{\theta} \Big)c \Big\|_{L^{4}}^{2}.
			\end{aligned}
		\end{equation*}
		Substituting the above estimates into (\ref{np 1}), we arrive at
		\begin{align*}
			&\frac{d}{dt}\|n^{p}\|_{L^{2}}^{2}+\frac{2(2p-1)}{Ap}\Big\|\Big(\partial_{r}, \frac{1}{r}\partial_{\theta} \Big)n^{p}\Big\|_{L^{2}}^{2}\\\leq&\frac{C(R)}{A}\Big(1+pD_{1}+p\mathcal{Q}_{1} \Big)\|n^{p}\|_{L^{2}}^{2}+\frac{Cp^{2}}{A}\|n^{p}\|_{L^{2}}\|\nabla n^{p}\|_{L^{2}}\Big\|\Big(1,\partial_{r},\frac{1}{r}\partial_{\theta} \Big)c \Big\|_{L^4}^{2}\\&+\frac{2p-1}{Ap}\Big\|\Big(\partial_{r},\frac{1}{r}\partial_{\theta} \Big)n^{p} \Big\|_{L^{2}}^{2}\\\leq&\frac{C(R)}{A}\Big(1+pD_{1}+p\mathcal{Q}_{1} \Big)\|n^{p}\|_{L^{2}}^{2}+\frac{Cp^{4}}{A}\|n^{p}\|_{L^{2}}^{2}\Big\|\Big(1,\partial_{r},\frac{1}{r}\partial_{\theta} \Big)c \Big\|_{L^{4}}^{4}\\&+\frac{5(2p-1)}{4Ap}\Big\|\Big(\partial_{r},\frac{1}{r}\partial_{\theta} \Big)n^{p} \Big\|_{L^{2}}^{2}.
		\end{align*}
		This implies that
		\begin{equation}\label{np 2}
			\begin{aligned}
				&\frac{d}{dt}\|n^{p}\|_{L^{2}}^{2}+\frac{1}{2A}\Big\|\Big(\partial_{r},\frac{1}{r}\partial_{\theta} \Big)n^{p} \Big\|_{L^{2}}^{2}\\\leq&\frac{C(R)p^{4}}{A}\|n^{p}\|_{L^{2}}^{2}\Big[\Big\|\Big(1,\partial_{r},\frac{1}{r}\partial_{\theta} \Big)c \Big\|_{L^{4}}^{4}+\mathcal{Q}_{1}+D_{1}+1 \Big].
			\end{aligned}
		\end{equation}
		Using the Nash inequality again
		\begin{equation*}
			\|n^{p}\|_{L^{2}}\leq C(R)\|n^{p}\|_{L^{1}}^{\frac12}\Big\|\Big(\partial_{r},\frac{1}{r}\partial_{\theta} \Big)n^{p} \Big\|_{L^{2}}^{\frac12},
		\end{equation*}
		we infer from (\ref{np 2}) that
		\begin{equation}\label{np 3}
			\begin{aligned}
				\frac{d}{dt}\|n^{p}\|_{L^{2}}^{2}\leq&-\frac{\|n^{p}\|_{L^{2}}^{4}}{2AC(R)\|n^{p}\|_{L^{1}}^{2}}+\frac{C(R)p^{4}}{A}\|n^{p}\|_{L^{2}}^{2}\Big[\Big\|\Big(1,\partial_{r},\frac{1}{r}\partial_{\theta} \Big)c \Big\|_{L^{4}}^{4}+\mathcal{Q}_{1}+D_{1}+1 \Big]
			\end{aligned}
		\end{equation}
		Using (\ref{assumption}), Lemma \ref{lem:ck}, Lemma \ref{lem:c0}, Lemma \ref{lem:n0 w0} and  Gagliardo-Nirenberg inequality, one obtains
		\begin{equation*}
			\begin{aligned}
				&\Big\|\Big(1,\partial_{r},\frac{1}{r}\partial_{\theta} \Big)c \Big\|_{L^{\infty}L^{4}}\leq\|(1,\partial_{r})\widehat{c}_{0}\|_{L^{\infty}L^{4}}+\sum_{k\neq 0, k\in\mathbb{Z}}\Big\|{\rm e}^{aA^{-\frac13}|k|^{\frac23}R^{-2}t}\Big(1,\partial_{r},\frac{|k|}{r} \Big)c_{k} \Big\|_{L^{\infty}L^{4}}\\\leq&C\|(1,\partial_{r})\widehat{c}_{0}\|_{L^{\infty}L^{2}}^{\frac34}\|(1,\partial_{r})\partial_{r}\widehat{c}_{0}\|_{L^{\infty}L^{2}}^{\frac14}\\+&C\Big(\sum_{k\neq 0, k\in\mathbb{Z}}\Big\|{\rm e}^{aA^{-\frac13}|k|^{\frac23}R^{-2}t}c_{k}\Big\|_{L^{2}} \Big)^{\frac12}\Big(\sum_{k\neq 0, k\in\mathbb{Z}}\|{\rm e}^{aA^{-\frac13}|k|^{\frac23}R^{-2}t}\Big(\partial_{r},\frac{|k|}{r} \Big)c_{k}\Big\|_{L^{2}} \Big)^{\frac12}\\+&C\Big(\sum_{k\neq 0, k\in\mathbb{Z}}\Big\|{\rm e}^{aA^{-\frac13}|k|^{\frac23}R^{-2}t}\Big(\partial_{r},\frac{|k|}{r} \Big)c_{k}\Big\|_{L^{2}} \Big)^{\frac12}\Big(\sum_{k\neq 0, k\in\mathbb{Z}}\|{\rm e}^{aA^{-\frac13}|k|^{\frac23}R^{-2}t}\Big(\partial_{r}^{2}, \frac{1}{r}\partial_{r},\frac{|k|^{2}}{r^{2}} \Big)c_{k}\Big\|_{L^{2}} \Big)^{\frac12}\\\leq&C(R)\big(\|\widehat{n}_{0}\|_{L^{\infty}L^{2}}+\sum_{k\neq 0, k\in\mathbb{Z}}\|{\rm e}^{aA^{-\frac13}|k|^{\frac23}R^{-2}t}n_{k}\|_{L^{\infty}L^{2}}\big)\leq C(R)\left(D_{1}+\mathcal{Q}_{1} \right).
			\end{aligned}
		\end{equation*}
		This along with (\ref{np 3}) gives that
		\begin{equation}\label{np 4}
			\begin{aligned}
				\frac{d}{dt}\|n^{p}\|_{L^{2}}^{2}\leq-\frac{\|n^{p}\|_{L^{2}}^{4}}{2AC(R)\|n^{p}\|_{L^{1}}^{2}}+\frac{C(R)p^{4}}{A}\|n^{p}\|_{L^{2}}^{2}\Big(D_{1}^{4}+\mathcal{Q}_{1}^{4}+1 \Big).
			\end{aligned}
		\end{equation}
		
		Claim that
		\begin{equation}\label{claim np}
			\begin{aligned}
				\sup_{t\geq 0}\|n^{p}\|_{L^{2}}^{2}\leq\max\Bigl\{4[C(R)]^{2}p^{4}\Big(D_{1}^{4}+\mathcal{Q}_{1}^{4}+1 \Big)\sup_{t\geq 0}\|n^{p}\|_{L^{1}}^{2}, 2\|n_{\rm in}^{p}\|_{L^{2}}^{2} \Bigr\}.
			\end{aligned}
		\end{equation}
		Otherwise, there  exists $t=t_{1}>0$ such that
		\begin{equation}\label{np t1}
			\|n^{p}(t_{1})\|_{L^{2}}^{2}=\max\Bigl\{4[C(R)]^{2}p^{4}\Big(D_{1}^{4}+\mathcal{Q}_{1}^{4}+1 \Big)\|n^{p}(t_{1})\|_{L^{1}}^{2}, 2\|n_{\rm in}^{p}\|_{L^{2}}^{2} \Bigr\}
		\end{equation}
		and
		\begin{equation}\label{np great 0}
			\frac{d}{dt}\Big(\|n^{p}(t)\|_{L^{2}}^{2} \Big)|_{t=t_{1}}\geq 0.
		\end{equation}
		Due to (\ref{np 4}) and (\ref{np t1}), there holds
		\begin{equation*}
			\begin{aligned}
				&\frac{d}{dt}\Big(\|n^{p}(t)\|_{L^{2}}^{2} \Big)|_{t=t_{1}}\\\leq&-\|n^{p}(t_{1})\|_{L^{2}}^{2}\left\{\frac{\|n^{p}(t_{1})\|_{L^{2}}^{2}}{2AC(R)\|n^{p}(t_{1})\|_{L^{1}}^{2}}-\frac{C(R)p^{4}}{A}\Big(D_{1}^{4}+\mathcal{Q}_{1}^{4}+1 \Big) \right\}\\=&-\|n^{p}(t_{1})\|_{L^{2}}^{2}\frac{C(R)p^{4}\Big(D_{1}^{4}+\mathcal{Q}_{1}^{4}+1 \Big)}{A}<0,
			\end{aligned}
		\end{equation*}
		which contradicts  with (\ref{np great 0}). Therefore, (\ref{claim np}) holds. 
		
		Next, the Moser-Alikakos iteration is used to determine $\mathcal{Q}_{2}$. Recall $p=2^{j}$ with $j\geq 1$, and  rewrite (\ref{claim np}) into 
		\begin{equation}\label{np 5}
			\begin{aligned}
				&\sup_{t\geq 0}\int_{0}^{2\pi}\int_{1}^{R}|n(t)|^{2^{j+1}}drd\theta\\\leq&\max\left\{Hp^{4}\left(\int_{0}^{2\pi}\int_{1}^{R}|n(t)|^{2^{j}}drd\theta \right)^{2}, 2\int_{0}^{2\pi}\int_{1}^{R}|n_{\rm in}|^{2^{j+1}}drd\theta \right\},
			\end{aligned}
		\end{equation}
		where $H=4[C(R)]^{2}\Big(D_{1}^{4}+\mathcal{Q}_{1}^{4}+1 \Big)$. By Lemma \ref{lem:n0 w0}, we have
		\begin{equation*}
			\|r^{\frac12}n_{0}\|_{L^{2}}\leq D_{1}.
		\end{equation*}
		Then
		\begin{equation*}
			\begin{aligned}
				\sup_{t\geq 0}\|n(t)\|_{L^{2}}\leq& 2\pi\|\widehat{n}_{0}\|_{L^{\infty}L^{2}}+\|\sum_{k\in\mathbb{Z},k\neq0} \widehat{n}_k\|_{L^{\infty}L^{2}}\\\leq& 2\pi\|r^{\frac12}\widehat{n}_{0}\|_{L^{\infty}L^{2}}+\sum_{k\neq 0, k\in\mathbb{Z}}\|n_{k}\|_{L^{\infty}L^{2}}\leq 2\pi D_{1}+\mathcal{Q}_{1}.
			\end{aligned}
		\end{equation*}
		Combining it with interpolation inequality, for $0<\theta<1$ and $j\geq 1$, we arrive at
		\begin{equation*}
			\|n_{\rm in}\|_{L^{2^{j}}}\leq\|n_{\rm in}\|_{L^{2}}^{\theta}\|n_{\rm in}\|_{L^{\infty}}^{1-\theta}\leq\|n_{\rm in}\|_{L^{2}}+\|n_{\rm in}\|_{L^{\infty}}\leq 2\pi D_{1}+\mathcal{Q}_{1}+\|n_{\rm in}\|_{L^{\infty}}.
		\end{equation*}
		This yields that
		\begin{equation*}
			\begin{aligned}
				2\int_{0}^{2\pi}\int_{1}^{R}|n_{\rm in}|^{2^{j+1}}drd\theta\leq 2\left(2\pi D_{1}+\mathcal{Q}_{1}+\|n_{\rm in}\|_{L^{\infty}} \right)^{2^{j+1}}\leq K^{2^{j+1}},
			\end{aligned}
		\end{equation*}
		where $K=2\left(2\pi D_{1}+\mathcal{Q}_{1}+\|n_{\rm in}\|_{L^{\infty}} \right)$. 
		Now, we rewrite (\ref{np 5}) as
		\begin{equation*}
			\sup_{t\geq 0}\int_{0}^{2\pi}\int_{1}^{R}|n(t)|^{2^{j+1}}drd\theta\leq\max\left\{H16^{j}\left(\sup_{t\geq 0}\int_{0}^{2\pi}\int_{1}^{R}|n(t)|^{2^{j}}drd\theta \right)^{2}, K^{2^{j+1}} \right\}.
		\end{equation*}
		For $j=k$, we get
		\begin{equation*}
			\sup_{t\geq 0}\int_{0}^{2\pi}\int_{1}^{R}|n(t)|^{2^{k+1}}drd\theta\leq H^{a_{k}}16^{b_{k}}K^{2^{k+1}},
		\end{equation*}
		where $a_{k}=1+2a_{k-1}$ and $b_{k}=k+2b_{k-1}$.
		
		Generally, one can obtain the following formulas
		\begin{equation*}
			a_{k}=2^{k}-1,\quad {\rm and}\quad b_{k}=2^{k+1}-k-2.
		\end{equation*}
		Thus, we arrive
		\begin{equation*}
			\sup_{t\geq 0}\left(\int_{0}^{2\pi}\int_{1}^{R}|n(t)|^{2^{k+1}}drd\theta \right)^{\frac{1}{2^{k+1}}}\leq H^{\frac{2^{k}-1}{2^{k+1}}}16^{\frac{2^{k+1}-k-2}{2^{k+1}}}K.
		\end{equation*}
		Letting $k\to\infty$, there holds
		\begin{equation}\label{prop2}
			\sup_{t\geq 0}\|n(t)\|_{L^{\infty}}\leq C(R)\Big(D_{1}^{4}+\mathcal{Q}_{1}^{4}+1 \Big)\left(2\pi D_{1}+\mathcal{Q}_{1}+\|n_{\rm in}\|_{L^{\infty}} \right)=:\mathcal{Q}_{2}.
		\end{equation}
		
		The proof is complete.
		
	\end{proof}
	
	\begin{corollary}\label{coro1}
		Under the assumptions of Theorem \ref{thm:main}, 
		when $A\geq A_{1}$, there holds 
		\begin{equation*}
			\begin{aligned}
				&\|u\|_{L^{\infty}L^{\infty}}\leq C(\|n_{\rm in}\|_{H^{1}}, \|u_{\rm in}\|_{H^{2}},R),\\
				&\|n\|_{L^{\infty}L^{\infty}}\leq C(\|n_{\rm in}\|_{H^{1}\cap L^{\infty}}, \|u_{\rm in}\|_{H^{2}},R).
			\end{aligned}
		\end{equation*}
	\end{corollary}
	
	\begin{proof}
		Rewriting the velocity $u$ into
		\begin{equation*}
			\begin{aligned}
				u(t,r,\theta)=\sum_{k\in\mathbb{Z}} \widehat{u}_k(t,r){\rm e}^{-ik\theta}
				=\widehat{u}_0(t,r)+\sum_{k\in\mathbb{Z},k\neq0} \widehat{u}_k(t,r){\rm e}^{-ik\theta}.
			\end{aligned}
		\end{equation*}
		Then, by 1D Gagliardo-Nirenberg inequality,  we obtain that 
		\begin{equation*}
			\begin{aligned}
				\|u\|_{L^{\infty}L^{\infty}}
				&\leq \|\widehat{u}_0\|_{L^{\infty}L^{\infty}}
				+\sum_{k\in\mathbb{Z},k\neq0} \|\widehat{u}_k\|_{L^{\infty}L^{\infty}}\\
				&\leq \|\widehat{w}_0\|_{L^{\infty}L^{2}}
				+\sum_{k\in\mathbb{Z},k\neq0} \|\widehat{w}_k\|_{L^{\infty}L^{2}}.
			\end{aligned}
		\end{equation*}
		Using \eqref{result D2} and \eqref{Ein}, 
		when $A\geq A_{1}$,
		we infer from the above inequality that 
		\begin{equation*}
			\begin{aligned}
				\|u\|_{L^{\infty}L^{\infty}}&\leq \|\widehat{w}_0\|_{L^{\infty}L^{2}}
				+\sum_{k\in\mathbb{Z},k\neq0} \|\widehat{w}_k\|_{L^{\infty}L^{2}}\\
				&\leq C(\|n_{\rm in}\|_{H^{1}}, \|u_{\rm in}\|_{H^{2}},R).
			\end{aligned}
		\end{equation*}
		By \eqref{Ein} and \eqref{prop2}, we obtain that   
		\begin{equation*}
			\begin{aligned}
				\|n\|_{L^{\infty}L^{\infty}}\leq C(\|n_{\rm in}\|_{H^{1}\cap L^{\infty}}, \|u_{\rm in}\|_{H^{2}},R).
			\end{aligned}
		\end{equation*}
		The proof is complete.
	\end{proof}

	\section*{Acknowledgement}
	%\addcontentsline{toc}{section}{References}
	W. Wang was supported by National Key R\&D Program of China (No.2023YFA1009200) and NSFC under grant 12471219.

	\section*{Declaration of competing interest}
	The authors declare that they have no known competing financial interests or personal relationships that could have appeared to influence the work reported in this paper.
	\section*{Data availability}
	No data was used in this paper.


\begin{thebibliography}{99}
		\bibitem{AHSL}Ameer G. A., Harmon W., Sasisekharan R. and Langer R. (1999). Investigation of a whole blood fluidized bed Taylor-Couette flow device for enzymatic heparin neutralization. {\em Biotechnol. Bioeng.}, 62, 602-8.
		\bibitem{AHL2024} An X., He T. and Li T. (2024). Nonlinear Asymptotic Stability and Transition Threshold for 2D Taylor-Couette Flows in Sobolev Spaces. {\em Commun. Math. Phys.}, 405:146.
		\bibitem{BJ1989}Beaudoin G. and Jaffrin M. Y. (1989). Plasma filtration in Couette flow membrane devices. {\em Int J Artif Organs}, 13, 43-51.
		\bibitem{Bedro2} Bedrossian J. and He S. (2017). Suppression of blow-up in Patlak-Keller-Segel via shear flows. {\em SIAM J. Math. Anal.}, 49(6), 4722-4766.
		\bibitem{BCM2008}Blanchet A., Carrillo J. A. and Masmoudi N. (2008). Infinite time aggregation for the critical
		Patlak-Keller-Segel model in  $\mathbb{R}^2$. {\em Comm. Pure Appl. Math.}, 61, 1449-1481.
%		\bibitem{BDDM2023}Buseghin F., Davila J., del Pino M. and Musso M. (2023). Existence of finite time blow-up in Keller-Segel system. arXiv:2312.01475.
		
		%		\bibitem{BDP2006}Blanchet A, Dolbeault J and Perthame B. (2006). Two-dimensional Keller-Segel model: Optimal critical mass
		%		and qualitative properties of the solutions. Electronic Journal of Differential Equations.
		%%		44(33): 1-33.
		%		\bibitem{CC2008}Calvez V. and Corrias L. (2008). The parabolic-parabolic Keller-Segel model in $\mathbb{R}^2$. {\em Commun. Math. Sci.}, 6(2), 417-447.
		\bibitem{CWY2025}Chen Y., Wang W. and Yang G. Quantitative blow-up suppression for the Patlak-Keller-Segel (-Navier-Stokes) system via Couette flow on $\mathbb{R}^2$. arXiv:2511.12915.
		\bibitem{CI1994}Chossat P. and Iooss G. (1994). The Couette-Taylor Problem. {\em Appl. Math. Sci.}, vol. 102, Springer Verlag, New York.
		\bibitem{CGMN2022}Collot C., Ghoul T., Masmoudi N. and Nguyen V.-T. (2022). Refined description and stability for singular solutions of the 2d Keller-Segel system. {\em Comm. Pure Appl. Math.}, 75(7), 1419-1516.
		\bibitem{CWW1}Cui S., Wang L. and Wang W. (2025). Suppression of blow-up in Patlak-Keller-Segel system coupled with linearized Navier-Stokes equations via the 3D Couette flow.  {\em J. Differential Equations}, 432, 113196.
		%\bibitem{CWW2025} Cui S., Wang L. and Wang W. Suppression of blow-up for the 3D Patlak-Keller-Segel-Navier-Stokes system via the Couette flow. arXiv:2412.19197.
		\bibitem{CWWT3} Cui S., Wang L., Wang W. and Wei J. On the sharp critical mass threshold for the 3D
		Patlak-Keller-Segel-Navier-Stokes system via Couette flow. arXiv:2506.10578.
		\bibitem{CWWTIT}Cui S., Wang L., Wang W., Wei J. and Yang G. (2025). Global bounded solutions of the 3D
		Patlak-Keller-Segel-Navier-Stokes system via Couette flow and logistic source. arXiv:2509.18718.
		\bibitem{cui1} Cui S. and Wang W. (2024). Suppression of blow-up in multi-species Patlak-Keller-Segel-Navier-Stokes system via the Poiseuille flow in a finite channel. {\em SIAM J. Math. Anal.}, 56(6), 7683-7712.
		%		\bibitem{wangweike1}Deng S., Shi B. and Wang W. (2025). Suppression of blow-up in 3-D Keller-Segel model via the Couette flow in whole space. {\em J. Differential Equations}, 432, 113265.
		%		\bibitem{DP2004}Dolbeault J. and Perthame B. (2004). Optimal critical mass in the two dimensional Keller-Segel model in $\mathbb{R}^2$.
		%		{\em C. R. Math.}, 339(9), 611-616.
		\bibitem{F2018}Falkovich G. (2018). Fluid Mechanics. Cambridge University Press, Cambridge.
		%		\bibitem{Feng1}Feng Y., Shi B. and Wang W. (2022). Dissipation enhancement of planar helical flows and applications to three-dimensional Kuramoto-Sivashinsky and Keller-Segel equations. {\em J. Differential Equations}, 313, 420-449.
		\bibitem{GG1999} Guillermo A., Gilda B., Ram S., William H., Charles L. and Robert L. (1999). Ex vivo evaluation of a Taylor-Couette flow, immobilized heparinase I device for clinical application.
		{\em Proc. Natl. Acad. Sci. USA}, 96, 2350-2355.
		\bibitem{he0}He S. (2018). Suppression of blow-up in parabolic-parabolic Patlak-Keller-Segel via strictly monotone shear flows. {\em Nonlinearity}, 31(8), 3651.
		\bibitem{he24-1}He S. (2025). Time-dependent Flows and Their Applications in Parabolic-parabolic Patlak-Keller-Segel Systems Part I: Alternating Flows. {\em J. Funct. Anal.}, 288(5), 110786.
		\bibitem{he24-2}He S. (2025). Time-dependent shear flows and their applications in parabolic-parabolic Patlak-Keller-Segel systems. {\em Nonlinearity}, 38(3), 035029.
		\bibitem{HT1}He S. and Tadmor E. (2021). Multi-species Patlak-Keller-Segel system. {\em Indiana Univ. Math. J.}, 70(4), 1577-1624.
		\bibitem{HP1}Hillen T. and Painter K. J. (2009). A user’s guide to PDE models for chemotaxis. {\em J. Math. Biol.}, 58(1), 183-217.
%		\bibitem{Hu2023}Hu Z. (2023). Suppression of Chemotactic Singularity via Viscous Flow with Large Buoyancy. {\em SIAM J. Math. Anal.}, 56(6), 7866-7902.
		\bibitem{Hu1}Hu Z. and Kiselev A. (2023). Suppression of chemotactic blow up by strong buoyancy in Stokes-Boussinesq flow with cold boundary. {\em J. Funct. Anal.}, 287(7), 110541.
		\bibitem{Hu0}Hu Z., Kiselev A. and Yao Y. (2025). Suppression of chemotactic singularity by buoyancy. {\em Geom. Funct. Anal.}, Vol. 35, 812-841.
		%		\bibitem{JL1992}Jager W. and Luckhaus S. (1992). On explosions of solutions to a system of partial differential equations modelling chemotaxis. {\em Trans. Amer. Math. Soc.}, 329(2), 819-824.
		\bibitem{KS1970}Keller E. and Segel L. (1970). Initiation of slime mold aggregation viewed as an instability.
		{\em J. Theor. Biol.}, 26(3), 399.
		\bibitem{KX2015}Kiselev A. and  Xu X. (2016). Suppression of chemotactic explosion by mixing. {\em Arch. Ration. Mech. Anal.}, 222(2), 1077-1112.
		\bibitem{K2003}K$\ddot{\rm o}$rfer S., Klaus S. and Mottaghy K. (2003). Application of Taylor vortices in hemocompatibility investigations. {\em Int J Artif Organs}, 26(4), 331-338.
		\bibitem{K1967}Kraichnan R. H. (1967). Inertial ranges in two-dimensional turbulence. {\em Phys. Fluids.}, 10, 1417-1423.
		\bibitem{Li0}Li H., Xiang Z. and Xu X. (2025). Suppression of blow-up in Patlak-Keller-Segel-Navier-Stokes system via the Poiseuille flow. {\em J. Differential Equations}, 434, 113301.
		\bibitem{MH1985}Mottaghy K. and Hanse H. J. (1985). Effect of combined shear, secondary and axial flow of blood on oxygen uptake. {\em Chem Eng Commun}, 36, 269-79.
		%		\bibitem{N1995}Nagai T. (1995). Blow-up of radially symmetric solutions to a chemotaxis system. {\em Adv. Math. Sci. Appl.}, 5, 581.
		%	   \bibitem{Na2000}Nagai T. (2000). Behavior of solutions to a parabolic-elliptic system modelling chemotaxis. {\em J. Korean Math. Soc.}, 37, 721-732.
		\bibitem{OY2001}Osaki K. and Yagi A. (2001). Finite dimensional attractor for one-dimensional Keller-Segel equations. {\em Funkc. Ekvacioj, Ser. Int.}, 44(3), 441-470.
		\bibitem{P1953}Patlak C. S. (1953).  Random walk with persistence and external bias. {\em Bull. Math. Biophys.}, 15(3), 311-338.
		\bibitem{P2000}Pope B. (2000). Turbulent Flows. Cambridge University Press, Cambridge.
		\bibitem{Schweyer1}Schweyer R. (2014). Stable blow-up dynamic for the parabolic-parabolic Patlak-Keller-Segel model. arXiv:1403.4975
		%			\bibitem{wangweike2} Shi B. and Wang W. (2024). Enhanced dissipation and blow-up suppression for the three dimensional Keller-Segel equation with the plane Couette-Poiseuille flow. {\em J. Differential Equations}, 403, 368-405.
		%		\bibitem{SW2019}Souplet P. and Winkler M. (2019). Blow-up profiles for the parabolic-elliptic Keller-Segel system in dimensions $n\geq3$. {\em Commun. Math. Phys.}, 367(2), 665-681.
		\bibitem{TW2016}Tao Y. and Winkler M. (2016). Blow-up prevention by quadratic degradation in a two-dimensional Keller-Segel-Navier-Stokes system. {\em Z. Angew. Math. Phys.}, 67, 1-23.
		\bibitem{Taylor1923}Taylor G. I. (1923). Stability of a viscous liquid contained between two rotating cylinders. {\em Proc. R. Soc. Lond. A.}, 102, 541-542.
%		\bibitem{Wanglili} Wang L., Wang W. and Zhang Y. Stability of a class of supercritical volume-filling chemotaxis-fluid model near Couette flow. arXiv:2410.02214.
		\bibitem{W2018}Wei D. (2018). Global well-posedness and blow-up for the 2-D Patlak-Keller-Segel equation. {\em J. Funct. Anal.}, 274(2), 388-401.
		\bibitem{Wink2013}Winkler M. (2013). Finite-time blow-up in the higher-dimensional parabolic-parabolic Keller-Segel system. {\em J. Math. Pures Appl.}, 100(5), 748-767.
		\bibitem{zeng}Zeng L., Zhang Z. and Zi R. (2021). Suppression of blow-up in Patlak-Keller-Segel-Navier-Stokes system via the Couette flow. {\em J. Funct. Anal.}, 280(10), 108-967.
	\end{thebibliography}
\end{document}